\documentclass[10pt,a4paper,english,french]{amsart}
\usepackage[T1]{fontenc}
\usepackage[latin9]{inputenc}
\usepackage{color}
\usepackage{graphicx}
\usepackage{amssymb}

\makeatletter
\numberwithin{equation}{section} 
\numberwithin{figure}{section} 
  \@ifundefined{theoremstyle}{\usepackage{amsthm}}{}
  \theoremstyle{plain}
  \newtheorem{thm}{Theorem}[section]
  \theoremstyle{remark}
  \newtheorem*{summary*}{Summary}
 \theoremstyle{definition}
  \newtheorem{example}[thm]{Example}
  \theoremstyle{plain}
  \newtheorem{prop}[thm]{Proposition}
  \theoremstyle{definition}
  \newtheorem{defn}[thm]{Definition}
  \theoremstyle{remark}
  \newtheorem{rem}[thm]{Remark}
  \theoremstyle{plain}
  \newtheorem{lem}[thm]{Lemma}
  \theoremstyle{plain}
  \newtheorem{cor}[thm]{Corollary}

\makeatother

\usepackage{babel}
\addto\extrasfrench{\providecommand{\fg}{\ifdim\lastskip>\z@\unskip\fi~\frqq}}

\begin{document}

\title{Sur le spectre semi-classique d'un système intégrable de dimension
1 autour d'une singularité hyperbolique}

\author{Olivier Lablée}

\date{17 Novembre 2008}

\maketitle
\begin{summary*}
Dans cette article on décrit le spectre semi-classique d'un opérateur
de Schrödinger sur $\mathbb{R}$ avec un potentiel type double puits.
La description qu'on donne est celle du spectre autour du maximun
local du potentiel. Dans la classification des singularités de l'application
moment d'un système intégrable, le double puits représente le cas
des singularités non-dégénérées de type hyperbolique.
\end{summary*}

\section{Introduction}

Sur la variété $M=\mathbb{R}$, l'opérateur de Schrödinger \foreignlanguage{english}{$P_{h}$}
de potentiel $V$, $V$ étant une fonction de $\mathbb{R}$ dans $\mathbb{R}$,
est l'opérateur linéaire non-borné sur l'espace des fonctions $\mathcal{C}^{\infty}$
à support compact $\mathcal{C}_{c}^{\infty}(\mathbb{R},\mathbb{R})$
définit par :\[
P_{h}:=-\frac{h^{2}}{2}\Delta+V,\]
où $V$ est l'opérateur de multiplication par la fonction $V$, le
laplacien est donné par $\Delta=\frac{d^{2}}{dx^{2}}$ et $h$ est
le paramètre semi-classique. Le spectre semi-classique d'un opérateur
de Schrödinger en dimension un est bien connu \textbf{{[}20]}, \textbf{{[}21]}
et \textbf{{[}11]} dans les zones dites elliptiques, c'est-à-dire
en dehors des maxima locaux de la fonction potentiel $V$; on parle
alors de règles de Bohr-Sommerfeld régulières. Dans cet article on
se concentre sur le cas d'un opérateur de Schrödinger avec un potentiel
type double puits, c'est-à-dire que $V\in\mathcal{C}^{\infty}(\mathbb{R})$
avec ${\displaystyle \lim_{|x|\rightarrow\infty}V(x)=+\infty}$ et
$V$ possédant exactement un maximum local non-dégénéré, que l'on
supposera par exemple atteint en $0$. Le modèle du double puits à
été beaucoup étudié \textbf{{[}3]}, \textbf{{[}22]} cependant son
spectre reste globalement assez mystérieux. Dans l'étude des singularités
de l'application moment d'un système complètement intégrable, l'opérateur
de Schrödinger avec double puits est le modèle type pour les singularités
non-dégénérées de type hyperbolique\textbf{ {[}34]}. En effet, pour
un hamiltonien $p\,:\, M\rightarrow\mathbb{R}$ tel que $0$ soit
valeur critique de $p$, et tel que les fibres dans un voisinage de
$0$ soient compactes et connexes et ne contiennent qu'un unique point
critique non-dégénéré de type hyperbolique: la fibre $\Lambda_{0}:=p^{-1}(0)$
est alors un ''huit'' et le feuilletage dans un voisinage de la
fibre singulière $\Lambda_{0}$ est difféomorphe à celui du double
puits.

\begin{center}
\includegraphics[scale=0.54]{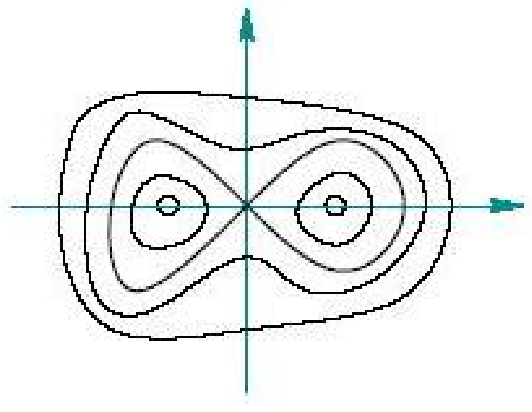}
\par\end{center}

\begin{center}
\textit{Fig. 1: Le feuilletage autour d'une singularité hyperbolique
dans l'espace de phases.}
\par\end{center}

\vspace{+0.25cm}

Dans une série de trois articles \textbf{{[}8]}, \textbf{{[}9]} et
\textbf{{[}10]} Y. Colin de Verdière et B. Parisse se sont intéressés
au spectre semi-classique de l'opérateur de Schrödinger, en dimension
1 avec un potentiel ayant un maximun local non-dégénéré. Dans \textbf{{[}10]}
ils traitent de manière générale l'étude des singularités. Dans \textbf{{[}8]}
et \textbf{{[}9]} les deux auteurs donnent une condition nécessaire
et suffisante pour trouver le spectre semi-classique dans un compact
de diamètre $h$ centré autour de l'origine de l'opérateur linéaire
: \[
P_{h}:=-\frac{h^{2}}{2}\frac{d^{2}}{dx^{2}}+V\]
 avec un potentiel $V$ type double puits. 

Dans la première partie de cet article, on fait quelques rappels sur
les outils semi-classiques. Dans la partie suivante, on rappelle la
formule donnée par Y. Colin de Verdière et B. Parisse. Dans la dernière
partie on utilise cette formule pour expliciter, dans une certaine
mesure, le spectre de l'opérateur ${\displaystyle P_{h}}$. On montre
en particulier que le spectre de l'opérateur $P_{h}$ dans le compact
$\left[-\sqrt{h},\sqrt{h}\right]$ est constitué de deux familles
de réels en quinconce, et que l'interstice spectral est de l'ordre
de $O(h/\left|\ln(h)\right|)$: voir le théorème 4.1.

\section{Préliminaires }

Dans toute cette section $X$ est une variété différentielle lisse
de dimension $n$ et on notera souvent par $M:=T^{*}X$ la variété
symplectique associée.

\subsection{Outils semi-classiques}

Pour le lecteur qui voudrait en savoir plus sur l'analyse semi-classique,
on conseille par exemple la littérature suivante : Y. Colin de Verdière
\textbf{{[}12]}, Dimassi-Sjöstrand \textbf{{[}14]}, L. Evans et M.
Zworski \textbf{{[}15]}, A. Martinez \textbf{{[}27]}, D. Robert \textbf{{[}30]},
S. Vu Ngoc \textbf{{[}34].} 

Sur la variété $X:=\mathbb{R}^{n}$, et pour $k,m\in\mathbb{Z}^{2}$,
on définit l'ensemble de symboles d'indice $k$ et de poids $\left\langle z\right\rangle ^{m}$
sur la variété $X$ où $\left\langle z\right\rangle =\left\langle x,\xi\right\rangle :=\left(1+|z|^{2}\right)^{\frac{1}{2}}$,
par :\[
S^{k}\left(X,\left\langle z\right\rangle ^{m}\right)\]
\[
:=\left\{ a_{h}(z)\in\mathcal{C}^{\infty}\left(T^{*}X\right),\,\forall\alpha\in\mathbb{N}^{n},\,\exists C_{\alpha}\geq0,\,\forall z\in T^{*}X,\,\left|{\displaystyle \partial_{z}^{\alpha}}a_{h}(z)\right|\leq C_{\alpha}h^{k}\left\langle z\right\rangle ^{m}\right\} .\]
De manière très formelle, la quantification de Weyl consiste à associer
à une fonction symbole $a_{h}:\,(x,\xi)\mapsto a_{h}(x,\xi)\in S^{k}\left(\left\langle z\right\rangle ^{m}\right)$
un opérateur linéaire $O_{p}^{w}(a_{h})$ de l'espace de Schwartz
$\mathcal{S}(X)$ dans lui même et admettant une représentation intégrale
: pour toute fonction $u\in\mathcal{S}(X)$ et pour tout $x\in X$:\[
O_{p}^{w}(a_{h})(u)(x):=\frac{1}{(2\pi h)^{n}}\int\int_{T^{*}X}e^{\frac{i}{h}(x-y)\xi}a\left(\frac{x+y}{2},\xi\right)u(y)\, dyd\xi.\]

\begin{example}
\textit{Le quantifié de Weyl de la fonction $(x,\xi)\mapsto x_{j}$
est l'opérateur de multiplication par la variable $x_{j}$. Le quantifié
de Weyl de la fonction $(x,\xi)\mapsto\xi_{j}$ est l'opérateur de
dérivation $-ih\frac{\partial}{\partial x_{j}}.$}

Rappelons aussi la :
\end{example}
\begin{prop}
Pour deux symboles $a_{h}$ et $b_{h}$ nous avons que :\textup{\[
O_{p}^{w}(a_{h}b_{h})=O_{p}^{w}(a_{h})\circ O_{p}^{w}(b_{h})+O(h)\]
\[
O_{p}^{w}\left\{ a_{h},b_{h}\right\} =\left[O_{p}^{w}(a_{h}),O_{p}^{w}(b_{h})\right]+O(h^{2})\]
}$\left\{ .,.\right\} $ étant les crochets de Poisson et $[.,.]$
le commutateur.
\end{prop}
Un opérateur linéaire $A$ est un opérateur pseudo-différentiel si
et seulement s'il existe un symbole $a_{h}$ tel que $A=O_{p}^{w}(a_{h}).$

En analyse semi-classique, on est aussi amené à considérer des symboles
ayant des développements asymptotiques en puissance de $h$: soit
$a_{h}\in S^{0}\left(X,\left\langle z\right\rangle ^{m}\right)$,
on dira que ce symbole est classique si et seulement s'il existe une
suite de symboles $\left(a_{j}\right)_{j\in\mathbb{N}}\in S^{0}\left(X,\left\langle z\right\rangle ^{m}\right)^{\mathbb{N}}$
indépendant de $h$ tels que pour tout $k^{\prime}\geq0$, on ait
:\[
\left(a_{h}(z)-{\displaystyle \sum_{j=0}^{k^{\prime}}a_{j}(z)h^{j}}\right)\in S^{k^{\prime}+1}\left(X,\left\langle z\right\rangle ^{m}\right).\]
On note alors $a_{h}={\displaystyle \sum_{j=0}^{+\infty}a_{j}h^{j}}$,
on dira aussi que $a_{0}$ est le symbole principal de $a_{h}$. 

Pour finir sur les symboles, on dit qu'un symbole $a\in S\left(X,\left\langle z\right\rangle ^{m}\right)$
est elliptique en $(x_{0},\xi_{0})\in T^{*}X$ si et seulement si
$|a(x_{0},\xi_{0})|\neq0$.

\subsection{Outils microlocaux}

De manière générale, sur une variété riemannienne $(X,g)$ complète
connexe, l'asymptotique du spectre de l'opérateur de Schrödinger\textit{\[
P_{h}:=-\frac{h^{2}}{2}\Delta_{g}+V\]
}ou plus généralement d'un opérateur pseudo-différentiel, est remarquablement
liée à une géométrie sous-jacente. Celle-ci vit sur le fibré cotangent
$T^{*}X$, vu comme une variété symplectique%
\footnote{Rappelons que le fibré cotangent d'une variété différentiable est
naturellement muni d'une structure symplectique. En effet, pour toute
variété $M$ lisse de dimension $n$, on peut munir de façon intrinsèque
son fibré cotangent $T^{*}M$ d'une structure de variété symplectique
$\left(T^{*}M,\omega\right)$ de dimension $2n$ définie par la différentielle
extérieure $\omega=d\alpha$ de la 1-forme de Liouville $\alpha$.%
}: c'est la géométrie de l'espace des phases. C'est d'ailleurs le même
phénomène qui permet de voir la mécanique classique (structure de
variété symplectique) comme limite de la mécanique quantique (structure
d'algèbre d'opérateurs). On est ainsi amené à définir une notion de
localisation dans l'espace des phases.

Donnons ici quelques éléments d'analyse microlocale, pour plus de
détails voir par exemple\textbf{ {[}33], {[}34]} ou \textbf{{[}12]}. 

Pour $h_{0}>0$ fixé, l'ensemble $A:=\left\{ \lambda(h)\in\mathbb{C}^{\left]0,h_{0}\right]},\,\exists N\in\mathbb{Z},\;\left|\lambda(h)\right|=O(h^{-N})\right\} $
est un anneau commutatif pour les opérations usuelles sur les fonctions.
On voit aussi sans peine que $I:=\left\{ \lambda(h)\in A,\,\lambda(h)=O(h^{\infty})\right\} $
est un idéal bilatère de $A$, on définit alors l'anneau $\mathbb{C}_{h}$
des constantes admissibles comme étant l'anneau quotient $A/I$.

On peut alors définir le $\mathbb{C}_{h}$-module des fonctions admissibles
:

\begin{defn}
L'ensemble $\mathcal{A}_{h}(X)$ des fonctions admissibles sur $X$
est l'ensemble des distributions $u_{h}\in\mathcal{D}^{\prime}(X)$
tels que pour tout opérateur pseudo-différentiel $P_{h}$ dont le
symbole dans une carte locale est a support compact\[
\exists N\in\mathbb{Z},\;\left\Vert P_{h}u_{h}\right\Vert _{L^{2}(X)}=O(h^{N}).\]
L'ensemble $\mathcal{A}_{h}(X)$ est un $\mathbb{C}_{h}$-module pour
les lois usuelles des fonctions. Un premier fait important est que
par le théorème de Calderon-Vaillancourt, on a l'inclusion : $L^{2}(X)\subset\mathcal{A}_{h}(X).$
\end{defn}
\begin{example}
Les fonctions WKB%
\footnote{Pour Wentzel, Kramers et Brillouin.%
} de la forme : \[
u_{h}(x)={\displaystyle \alpha(x)e^{i\frac{S(x)}{h}}}\]
$S$ étant une fonction réelle $\mathcal{C}^{\infty}$, sont des fonctions
admissibles stables par l'action d'un opérateur pseudo-différentiel.
\end{example}
A tout élément $u_{h}$ du $\mathbb{C}_{h}$-module des fonctions
admissibles est associé un sous-ensemble de $T^{*}X$, cet ensemble,
nommé micro-support%
\footnote{Ou front d'ondes.%
} décrit la localisation de la fonction $u_{h}$ dans l'espace des
phases.

\begin{defn}
Soit $u_{h}\in\mathcal{A}_{h}(X)$, on dira que $u_{h}$ est négligeable
au point $m\in T^{*}X$, si et seulement s'il existe $P_{h}$ un opérateur
pseudo-différentiel elliptique en $m$ tels que :

\[
\left\Vert P_{h}u_{h}\right\Vert _{L^{2}(X)}=O(h^{\infty}).\]

\end{defn}
On définit alors $MS(u_{h})$, le micro-support de $u_{h}$ comme
le complémentaire dans $T^{*}X$ de l'ensemble des points $m\in T^{*}X$
où $u_{h}$ est négligeable. Parmi les propriétés liées au micro-support
nous avons que si $P_{h}$ est un opérateur pseudo-différentiel de
symbole principal $p$ alors on a l'implication : \[
P_{h}u_{h}=O(h^{\infty})\Rightarrow MS(u_{h})\subset p^{-1}(0).\]
Donc si par exemple $P_{h}$ est un opérateur de symbole principal
$p$, $\lambda$ un scalaire, et si $u_{h}$ est une fonction non
nulle telle que : $\left(P_{h}-\lambda I_{d}\right)u_{h}=O(h^{\infty})$
alors $MS(u_{h})\subset p^{-1}(\lambda).$ Ceci est une propriété
fondamentale de l'analyse microlocale: elle donne une localisation
des fonctions propres dans l'espace des phases.

\begin{example}
Pour une fonction WKB : $u_{h}(x)={\displaystyle {\displaystyle \alpha(x)e^{i\frac{S(x)}{h}}}}$
on a : \[
MS(u_{h})=\left\{ \left(x,dS(x)\right),\alpha(x)\neq0\right\} .\]

\end{example}
\begin{defn}
Soient $u_{h},v_{h}\in\mathcal{A}_{h}(X)^{2}$, on dira que $u_{h}=v_{h}+O(h^{\infty})$
sur un ouvert $U\subset T^{*}X$ si et seulement si :\[
MS(u_{h}-v_{h})\cap U=\textrm{\O}.\]

\end{defn}
Avec les propriétés du micro-support, on peut montrer que pour tout
ouvert $U$ de $T^{*}X$, l'ensemble $\left\{ u_{h}\in\mathcal{A}_{h}(X)/MS(u_{h})\cap U=\textrm{\O}\right\} $
est un $\mathbb{C}_{h}-$sous-module de $\mathcal{A}_{h}(X)$, on
peut alors définir l'espace des micro-fonctions :

\begin{defn}
Soit $U$ un ouvert non vide de $T^{*}X$, on définit l'espace des
micro-fonctions sur $U$ comme étant le $\mathbb{C}_{h}-$module quotient:\[
\mathcal{M}_{h}(U):=\mathcal{A}_{h}(X)/\left\{ u_{h}\in\mathcal{A}_{h}(X),\, MS(u_{h})\cap U=\textrm{\O}\right\} .\]

\end{defn}
Les opérateurs pseudo-différentiels agissent sur $\mathcal{M}_{h}(U)$,
en effet : pour tout opérateur pseudo-différentiel $P_{h}$ on a $MS(P_{h}u_{h})\subset MS(u_{h})$
et ainsi $P_{h}\left(\mathcal{M}_{h}(U)\right)\subset\mathcal{M}_{h}(U).$

A tout triplet $\left(P_{h},\lambda,U\right)$ où $P_{h}$ est un
opérateur pseudo-différentiel, $\lambda$ un scalaire de l'anneau
$\mathbb{\mathbb{C}}_{h}$ et $U$ un ouvert non vide de $T^{*}X$,
on peut associer l'ensemble $\mathcal{L}\left(P_{h},\lambda,U\right)$
des microfonctions $u_{h}$ solutions dans l'ouvert $U$ de $\left(P_{h}-\lambda I_{d}\right)u_{h}=O(h^{\infty})$.
L'ensemble $\mathcal{L}\left(P_{h},\lambda,U\right)$ est un $\mathbb{C}_{h}$-module,
et si $\Omega$ désigne un ensemble d'indices quelconque, la famille
d'ensembles $\left\{ \mathcal{L}\left(P_{h},\lambda,U_{x}\right),\, x\in\Omega\right\} $
est un faisceau au dessus de ${\displaystyle \bigcup_{x\in\Omega}}U_{x}$.
En effet toute solution peut être restreinte sur des ouverts plus
petits d'une unique manière, et deux solutions $u_{h}$ définie sur
un ouvert $U_{x}$ et $v_{h}$ définie sur un autre ouvert $U_{y}$
et telles que $u_{h}=v_{h}$ sur l'ouvert $U_{x}\cap U_{y}$ peuvent
être misent ensemble pour former une solution globale sur l'ouvert
$U_{x}\cup U_{y}$. Ce faisceau est supporté%
\footnote{Au sens du micro-support.%
} sur l'ensemble $p^{-1}(\lambda)\subset T^{*}X$.

\subsection{Systèmes intégrables semi-classiques}

Un système intégrable classique est la donnée d'une variété symplectique
$\left(M,\omega\right)$ de dimension $2n$ et de $n$ fonctions $\left(f_{1},...,f_{n}\right)$
de l'algèbre $\mathcal{C}^{\infty}(M)$ telles que les différentielles
$\left(df_{i}(x)\right)_{i=1,...,n}$ sont libres presque partout
sur $M$, et telles que pour tout indices $i,j$ on ait $\left\{ f_{i},f_{j}\right\} =0$.
On définit aussi l 'application moment classique associée:\[
\mathbf{f}\,:\left\{ \begin{array}{cc}
M\rightarrow\mathbb{R}^{n}\\
\\x\mapsto\left(f_{1}(x),...,f_{n}(x)\right).\end{array}\right.\]
Un système intégrable semi-classique sur une variété $X$ est la donné
de $n$ opérateurs pseudo-différentiels $P_{1},...,P_{n}$ sur $L^{2}(X)$
tels que pour tout indices $i$ et $j$ on ait $\left[P_{i},P_{j}\right]=O(h^{\infty})$
et dont les symboles principaux forment un système intégrable sur
$M:=T^{*}X$. On notera par $\mathbf{P}:=(P_{1},...,P_{n})$ l'application
moment quantique et par $\mathbf{p}:=(p_{1},..,.p_{n})$ l'application
moment classique associée aux symboles principaux de $\mathbf{P}.$
Les points réguliers de l'application moment classique sont les points
$m\in M$ tels que les différentielles $\left(dp_{i}(m)\right)_{i=1,...,n}$
sont libres. Les points réguliers d'un système intégrable ont une
description symplectique locale simple donnée par le théorème de Darboux-Carathéodory
(voir par exemple \textbf{{[}2]}). Une fibre $\Lambda_{c}:=\mathbf{p}^{-1}(c)$
est régulière si et seulement si tous les points de $\Lambda_{c}$
sont réguliers pour $\mathbf{p}.$ Les fibres régulières sont décrites
par le théorème actions-angles, nommé aussi théorème d'Arnold-Liouville-Mineur,
qui donne la dynamique classique au voisinage d'une fibre régulière
connexe et compacte : le flot hamiltonien associé à une intégrale
première est quasi-périodique (droites s'enroulant sur un tore). Ces
deux théorèmes ont un analogue semi-classique : Y. Colin de Verdière
pour le théorème de Darboux-Carathéodory \textbf{{[}7]} et S. Vu Ngoc
pour le théorème actions-angles \textbf{{[}31]},\textbf{{[}32]}. Le
théorème de Darboux-Carathéodory semi-classique permet de faire une
description précise de l'ensemble des micro-solutions des équations
$P_{j}u_{h}=O(h^{\infty})$. Les résultats d'analyse microlocale nous
informe déjà que les solutions $u_{h}$ sont localisées sur ${\displaystyle \bigcap_{j=1}^{n}}p_{j}^{-1}(0)$;
mais en fait on a bien mieux :

\begin{prop}
Pour tout point $m\in M$ régulier de $\mathbf{p}=(p_{1},...p_{n})$
tel que $\mathbf{p}(m)=0$, le faisceau des microsolutions de l'équation
:\[
P_{j}u_{h}=O(h^{\infty})\textrm{ près de }m\]
est un faisceau en $\mathbb{C}_{h}$-module libre de rang $1$ engendré
par $U^{-1}(\mathbf{1})$ où $U$ est un opérateur intégral de Fourier
et $\mathbf{1}$ est une microfonction égale à $1$ près de l'origine. 
\end{prop}

\subsection{Fibres régulières d'un système intégrable}

En conséquence pour une fibre \foreignlanguage{english}{$\Lambda_{E}:=\mathbf{p}^{-1}(E)$
compacte, connexe et régulière,} toute microsolution $u_{h}$ de $\left(P_{h}-EI_{d}\right)u_{h}=O(h^{\infty})$
est engendré par \foreignlanguage{english}{$U^{-1}(\mathbf{1})$.
La théorie des opérateurs intégraux de Fourier montre que $u_{h}$
est nécesseraiment du type WKB. On va maintenant décrire comment on
prolonge une microsolution d'un ouvert à un autre le long d'une fibre
régulière (pour plus de détails, voir \textbf{{[}34]}). Pour commencer
on se donne un recouvrement fini $\left(U_{\alpha}\right)_{\alpha\in\Omega}$
d'ouverts de la fibre $\Lambda_{E}$. Pour tout couple d'ouverts non
vides $U_{\alpha}$ et $U_{\beta}$ du recouvrement tels que $U_{\alpha}\cap U_{\beta}$
est non vide et connexe; si on considère alors deux microfonctions
$\varphi_{\alpha}$ et $\varphi_{\beta}$ solutions de $\left(P_{h}-\lambda I_{d}\right)u_{h}=O(h^{\infty})$
microlocalement sur les ouverts respectifs $U_{\alpha}$ et $U_{\beta}$,
les microfonctions $\varphi_{\alpha}$ et $\varphi_{\beta}$} sont
respectivement engendrés par \foreignlanguage{english}{$U^{-1}\left(\mathbf{1}_{\alpha}\right)$
et par $U^{-1}\left(\mathbf{1}_{\beta}\right)$, $\mathbf{1}_{\alpha}$
et $\mathbf{1}_{\beta}$ étant égale à $1$ microlocalement sur $U_{\alpha}$
et respectivement sur $U_{\beta}$. En se plaçant sur $U_{\alpha}\cap U_{\beta}$
et en utilsant l'argument de la dimension $1$ on a l'existence de
$c_{\alpha,\beta}\in\mathbb{C}_{h}$ tel que sur $U_{\alpha}\cap U_{\beta}$
on ait\[
U^{-1}\left(\mathbf{1}_{\alpha}\right)=c_{\alpha,\beta}U^{-1}\left(\mathbf{1}_{\beta}\right)\]
et donc, sur $U_{\alpha}\cap U_{\beta}$ on a : \[
\mathbf{1}_{\alpha}=c_{\alpha,\beta}\mathbf{1}_{\beta}.\]
La théorie des opérateurs intégraux de Fourier montre (voir \textbf{{[}33]},\textbf{{[}34]})
que la constante $c_{\alpha,\beta}$ s'écrit sous la forme \[
c_{\alpha,\beta}=e^{\frac{iS_{\alpha\beta}}{h}}\]
le scalaire $S_{\alpha\beta}$ étant dans $\mathbb{C}_{h}$ est dépendant
de la variable $E$. Plus généralement pour une famille finie $\left(U_{k}\right)_{k=1,...,l}$
d'ouverts non vides recouvrant une partie compacte et connexe de la
fibre régulière $\Lambda_{E}$ telle que pour tout indice $k\in\left\{ 1,..,l-1\right\} ,$
$U_{k}\cap U_{k+1}$ est non vide et connexe. Sur chaques ouverts
$U_{k}$ on a un générateur $\mathbf{1}_{k}$ de $\mathcal{L}\left(P_{h},\lambda,U_{k}\right)$
et pour tout indice $k\in\left\{ 1,..,l-1\right\} $ il existe $c_{k,k+1}=e^{\frac{iS_{k,k+1}}{h}}\in\mathbb{C}_{h}$
tel que :}

\selectlanguage{english}%
\[
\mathbf{1}_{k}=c_{k,k+1}\mathbf{1}_{k+1}\]
\foreignlanguage{french}{ainsi nous avons alors l'égalité suivante
}\[
\mathbf{1}_{1}=c_{1,2}c_{2,3}\ldots c_{l-1,l}\mathbf{1}_{l}.\]
On peut donc écrire $\mathbf{1}_{1}=e^{\frac{iS_{1,l}}{h}}1_{l}$
où on a posé $S_{1,l}=\sum_{k=1}^{l-1}S_{k,k+1}$. \foreignlanguage{french}{La
dépendance en la variable $E$ est lisse : les fonctions $E\mapsto S_{.}(E)$
sont $\mathcal{C}^{\infty}$(voir \textbf{{[}31]},\textbf{{[}33]}
et \textbf{{[}34]}).}

\selectlanguage{french}%
\begin{center}
\includegraphics[scale=0.5]{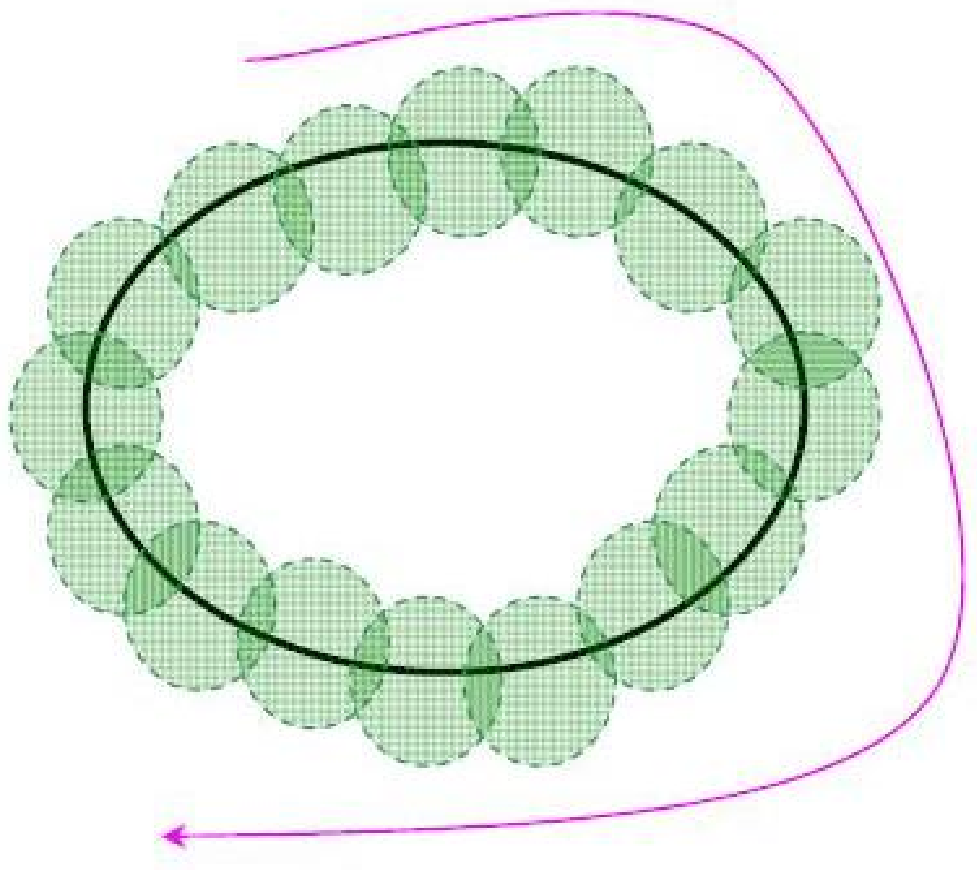}
\par\end{center}

\begin{center}
\textit{Fig. 2: Un recouvrement par des ouverts d'une fibre régulière.}
\par\end{center}

\vspace{+0.25cm}

\subsection{Théorème d'Egorov et opérateurs intégraux de Fourier}

Pour finir donnons le théorème d'Egorov qui permet de définir rapidement
la notion d'opérateur intégral de Fourier, voir par exemple \textbf{{[}12]}
:

\begin{thm}
\textbf{(Egorov)} : Soient $\left(T^{*}X,d\alpha\right)$ et $\left(T^{*}Y,d\beta\right)$
deux variétés symplectomorphe : il existe $\chi$ un symplectomorphisme
de $T^{*}X$ dans $T^{*}Y$. On supposera que $\chi$ est exact :
$\chi^{*}\beta-\alpha$ est une 1-forme exacte sur $X$. Alors il
existe $\widetilde{\chi}$ un morphisme de $\mathbb{C}_{h}$-module
de $\mathcal{M}_{h}(X)$ dans $\mathcal{M}_{h}(Y)$ inversible tel
que pour tout $a\in\mathcal{M}_{h}(Y)$, en notant par $\hat{a}=O_{p}^{w}(a)$,
l'opérateur :\[
B=\widetilde{\chi}^{-1}\circ\hat{a}\circ\widetilde{\chi}\]
est un opérateur pseudo-différentiel sur $\mathcal{M}_{h}(X)$, et
dont le symbole principal est donné par $a_{0}\circ\chi$, $a_{0}$
étant le symbole principal de $\hat{a}$. On dit que $\widetilde{\chi}$
est un opérateur intégral de Fourier associé à $\chi$.
\end{thm}

\subsection{Théorie spectrale de l'opérateur de Schrödinger}

Pour un panorama à la fois historique et actuel sur l'étude du laplacien
et de l'opérateur de Schrödinger sur une variété riemannienne on pourra
consulter\textbf{ {[}23]}. Ici on considère l'opérateur linéaire sur
$L^{2}(\mathbb{R})$:

\[
{\displaystyle P_{h}}:=-\frac{h^{2}}{2}\frac{d^{2}}{dx^{2}}+V.\]
T. Carleman \textbf{{[}4]} en 1934 à montré que si la fonction $V$
est localement bornée et globalement minorée, alors l'opérateur de
Schrödinger ${\displaystyle P_{h}}$ est essentiellement auto-adjoint.
En 1934 K. Friedrichs \textbf{{[}16]} a montré que dans le cas où
la fonction $V$ est confinante, ie ${\displaystyle \lim_{|x|\rightarrow\infty}V(x)=+\infty}$
alors le spectre de l'opérateur de Schrödinger ${\displaystyle P_{h}}$
est constitué d'une suite de valeurs propres de multiplicité finie
s'accumulant en $+\infty$. Le théorème de Courant de 1953 \textbf{{[}13]},
assure en particulier que la première valeur propre $\mu_{1}(h)$
de l'opérateur ${\displaystyle P_{h}}$ est simple :\[
\min_{x\in\mathbb{R}}V(x)\leq\mu_{1}(h)<\mu_{2}(h)\leq\cdots\leq\mu_{n}(h){\displaystyle \rightarrow}+\infty.\]
Rappelons \textbf{{[}34]} que pour un compact $K$ de $\left[\min(V),+\infty\right[$
on a la :

\begin{defn}
On appelle spectre semi-classique dans le compact $K$, l'ensemble
$\Sigma_{h}({\displaystyle P_{h}},K)$ des familles de réels $E_{h}\in\mathbb{R}$
vérifiant ${\displaystyle \lim_{h\rightarrow0}}E_{h}\rightarrow E\in K$
et telles qu'il existe une microfonction $u_{h}$ avec $MS(u_{h})=p^{-1}(E)$
et vérifiant :\[
\left({\displaystyle P_{h}}-E_{h}\right)u_{h}=O(h^{\infty}).\]
On appelle multiplicité microlocale de $E_{h}$ la dimension du $\mathbb{C}_{h}$-module
des solutions microlocales de cette équation.
\end{defn}
Moralement le spectre semi-classique (ou microlocal) correspond aux
valeurs propres approchées avec une précision d'ordre $O(h^{\infty})$
incluant les multiplicités. Le spectre semi-classique et le spectre
exact sont liés par la :

\begin{prop}
\textbf{{[}34]} Sur un compact $K$ de $\mathbb{R}$, le spectre semi-classique
$\Sigma_{h}({\displaystyle P_{h},K})$ et le spectre exact $\sigma({\displaystyle P_{h}})$
de l'opérateur linéaire auto-adjoint ${\displaystyle P}_{h}$ sont
liés par \textbf{:\[
\Sigma_{h}({\displaystyle P_{h},K})=\sigma({\displaystyle P})\cap K+O(h^{\infty})\]
}au sens où si $\lambda_{h}\in\Sigma_{h}(P_{h},K)$, alors il existe
$\mu_{h}\in\sigma(P_{h})\cap K$ tel que $\lambda_{h}=\mu_{h}+O(h^{\infty})$;
et si $\mu_{h}\in\sigma({\displaystyle P}_{h})\cap K$, alors il existe
$\lambda_{h}\in\Sigma_{h}({\displaystyle P_{h}},K)$ tel que $\mu_{h}=\lambda_{h}+O(h^{\infty}).$
De plus pour toute famille $E_{h}$ ayant une limite finie $E\in K$
lorsque $h\rightarrow0$, si la multiplicité microlocale de $E_{h}$
est bien définie et est finie, alors elle est égale pour $h$ assez
petit au rang du projecteur spectral de $P_{h}$ sur une boule de
diamètre $O(h^{\infty})$ centrée autour de $E_{h}$.
\end{prop}

\section{La formule de Colin de Verdière-Parisse}

\subsection{Le cadre}

Soit $V\in\mathcal{C}^{\infty}(\mathbb{R})$ telle que ${\displaystyle \lim_{|x|\rightarrow\infty}V(x)=+\infty}$
et $V$ possédant exactement un maximun local non dégénéré, que l'on
supposera atteint en $0$, ainsi : $V(0)=0,\, V'(0)=0,\, V''(0)<0$.

\begin{example}
Un exemple typique non-générique est la fonction $V(x)=x^{4}-x^{2}$.
\end{example}
\begin{center}
\includegraphics[scale=0.35]{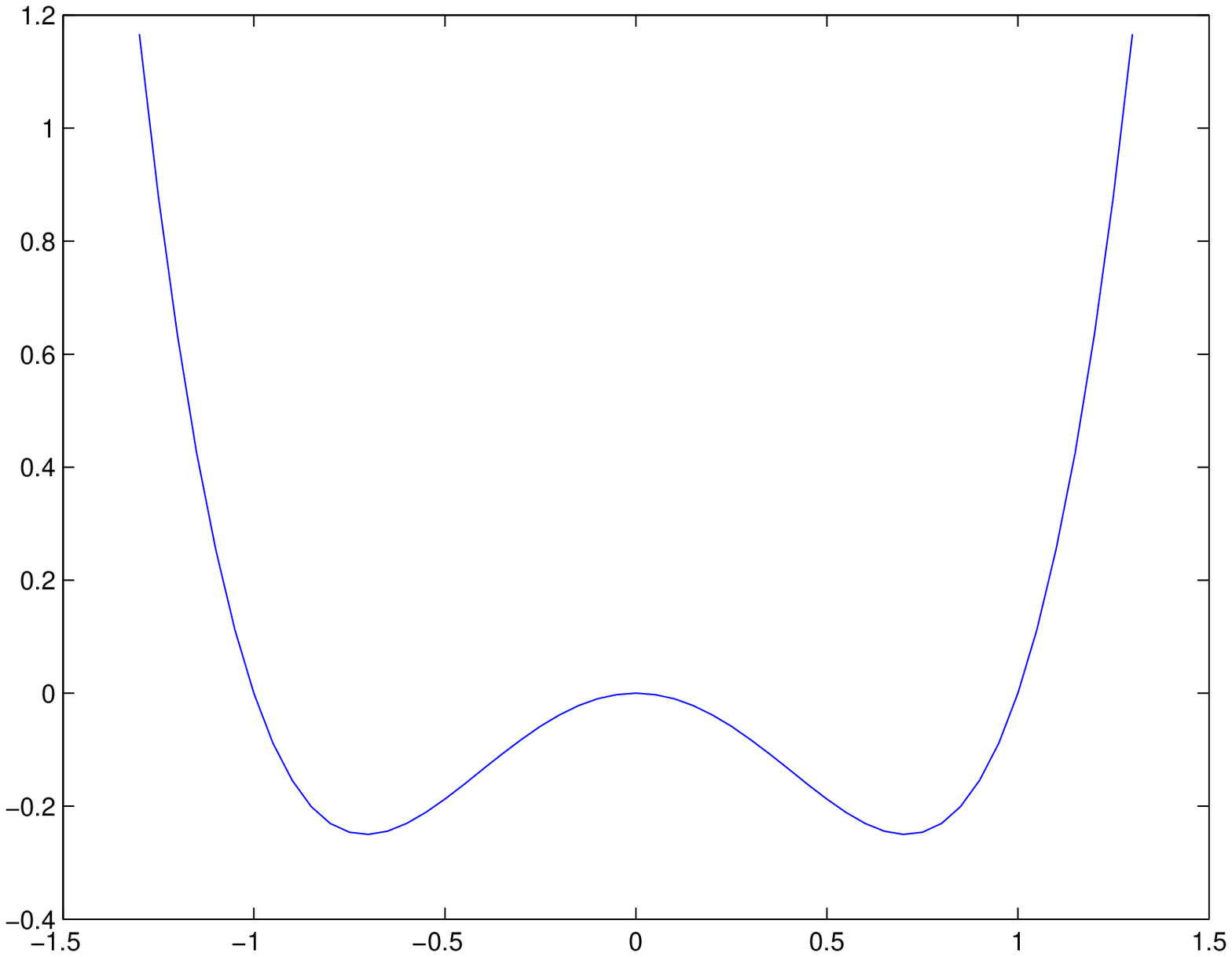}
\par\end{center}

\begin{center}
\textit{Fig. 3: La courbe représentative de la fonction potentiel
paire $V(x)=x^{4}-x^{2}$. On distingue les deux puits (les minima)
du potentiel, le droit et le gauche.}
\par\end{center}

\vspace{+0.25cm}

On notera par $p$ la fonction définie sur le fibré cotangent de $\mathbb{R}$
par : \[
p(x,\xi):=\frac{\xi^{2}}{2}+V(x)\in\mathcal{C}^{\infty}(T^{*}\mathbb{R},\mathbb{R}).\]

\begin{center}
\includegraphics[scale=0.35]{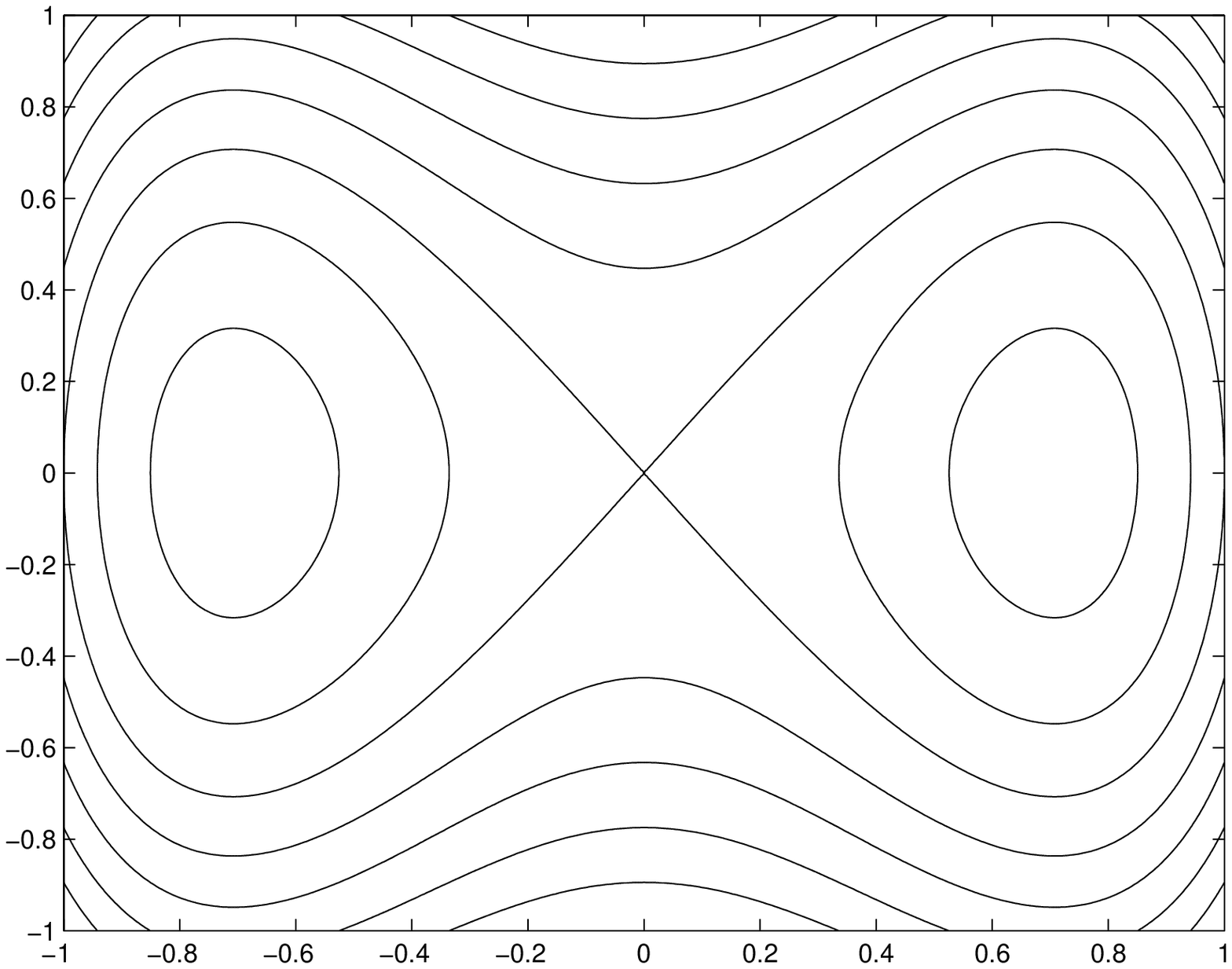}
\par\end{center}

\begin{center}
\textit{Fig. 4: Le feuilletage de $p^{-1}(c)$; avec des c$>0$ les
fibres elliptiques ont une seule composante connexe, pour c=0 : la
fibre singulière en forme de huit hyperbolique, et pour des $c<0$
: les fibres elliptiques ont deux composantes connexes.}
\par\end{center}

\vspace{+0.25cm}

Son quantifié de Weyl ${\displaystyle P_{h}}$ est donné par :\[
{\displaystyle P_{h}}=-\frac{h^{2}}{2}\Delta+V.\]
Pour étudier le spectre de l'opérateur ${\displaystyle P_{h}}$ dans
une fenêtre de taille $E$, avec $E\in\left[-1,1\right]$, considérons
alors l'opérateur :

\[
{\displaystyle P_{h}}-EI_{d}\]
 $I_{d}$ étant l'opérateur identité. Ainsi par définition nous avons
que\[
\left({\displaystyle P_{h}}-EI_{d}\right)u_{h}=O(h^{\infty})\Leftrightarrow E\in\Sigma_{h}(P_{h}).\]

\subsection{Énoncé de la formule}

Y. Colin de Verdière et B. Parisse ont donné les règles de Bohr-Sommerfeld
dans le cas singulier sous la forme suivante :

\begin{thm}
Pour $E\in[-1,1]$ l'équation : \[
\left(P_{h}-EI_{d}\right)u_{h}=O(h^{\infty})\]
admet une solution $u_{h}\in L^{2}(\mathbb{R})$ non triviale avec
son microsupport $MS(u_{h})=p^{-1}\{E\}$ si et seulement si $E$
vérifie l'équation suivante:\begin{equation}
\frac{1}{\sqrt{1+e^{\frac{2\pi\varepsilon}{h}}}}\cos\left(\frac{\theta_{+}-\theta_{-}}{2}\right)=\cos\left(-\frac{\theta_{+}+\theta_{-}}{2}+\frac{\pi}{2}+\frac{\varepsilon}{h}\ln(h)+\arg\left(\Gamma\left(\frac{1}{2}+i\frac{\varepsilon}{h}\right)\right)\right)\label{eq:}\end{equation}
où :\[
\varepsilon:=\varepsilon(E),\,\theta_{+/-}:=\theta_{+/-}(E)=S^{+/-}(E)/h.\]
Les fonctions $\varepsilon$ et $S^{+/-}$ admettant des développements
asymptotiques en puissance de $h$ avec des coefficients $\mathcal{C}^{\infty}$
par rapport à $E$. 
\end{thm}

\subsection{Les grandes étapes de la preuve}

On résume \textbf{{[}8]}, \textbf{{[}9]} et une bonne partie de \textbf{{[}10]}.
La preuve de la formule se décompose en plusieurs grandes étapes.

\subsection*{La stratégie }

La première étape de la preuve est une étude locale autour de la singularité.
Pour ça on utilise une forme normale de Birkhoff quantique de manière
à se ramener à une équation différentielle linéaire du premier ordre.
On exhibe alors quatre solutions et on utilise le fait que l'ensemble
des solutions est de dimension 2, pour en déduire une dépendance linéaire
entres ces solutions. La seconde étape consiste à prolonger de manière
globale les fonctions solutions, ce qui donnera à nouveau une dépendance
linéaire entres les solutions. A la fin, on exprime simultanément
ces relations linéaires avec un déterminant.

\subsection*{Première étape : Étude locale autour de la singularité}

Pour un réel $E\in[-1,1]$, on va étudier l'équation $\left(P_{h}-EI_{d}\right)u_{h}=O(h^{\infty})$
avec une forme normale quantique autour de l'origine, utilisons le
:

\begin{thm}
\textbf{(Théorème 3 de {[}10])} Il existe ${\displaystyle U}$ un
opérateur intégral de Fourier, ${\displaystyle N}$ un opérateur pseudo-différentiel
elliptique en $0$ et une fonction $\varepsilon$ ayant un développement
asymptotique en puissance de $h$: $\varepsilon(E)=\sum_{j\geq0}\varepsilon_{j}(E)h^{j}$
où les fonctions $\varepsilon_{j}$ sont de classe $\mathcal{C}^{\infty}$
par rapport à $E$ et indépendante de $h$, tels que microlocalement
dans un ouvert $\Omega_{0}$ contenant l'origine, on ait pour tout
$E\in[-1,1]$ :\begin{eqnarray*}
U^{-1}{\displaystyle \left(P_{h}-EI_{d}\right)}{\displaystyle U}={\displaystyle N}\left({\displaystyle \widehat{x\xi}}-\varepsilon(E)I_{d}\right)\end{eqnarray*}
où\[
{\displaystyle \widehat{x\xi}}:=\frac{h}{i}\left(x\frac{d}{dx}+\frac{1}{2}I_{d}\right)\]
avec $\varepsilon_{0}(0)=0$ et $\varepsilon_{0}^{\prime}(0)=\frac{1}{\sqrt{-V^{''}(0)}}$. 
\end{thm}
La démonstration de ce théorème de forme normale quantique est donnée
dans \textbf{{[}8]} ou \textbf{{[}10]}, la preuve utilise le lemme
de Morse isochore \textbf{{[}6]. }

\begin{rem}
Cette forme normale reste valide uniquement pour $|E|\leq1$.
\end{rem}
Pour tout $|E|\leq1$ on a :\[
\varepsilon(E)=\varepsilon_{0}(E)+\sum_{j=1}^{\infty}\varepsilon_{j}(E)h^{j}.\]
Ainsi, en appliquant la formule de Taylor sur la fonction lisse $\varepsilon_{0}$,
on a pour tout $E\in[-1,1]$ nous avons :\[
\varepsilon(E)=\frac{E}{\sqrt{-V^{''}(0)}}+O(E^{2})+\sum_{j=1}^{\infty}\varepsilon_{j}(E)h^{j}.\]
Par la suite on va utiliser ce théorème avec $E:=\lambda h^{\alpha}$
où $\lambda\in[-1,1]$ et $\alpha\geq0$. Ainsi dans ce là nous avons
pour tout $\lambda\in[-1,1]$ :

\begin{equation}
\varepsilon(\lambda h^{\alpha})=\frac{\lambda h^{\alpha}}{\sqrt{-V^{''}(0)}}+O(h^{2\alpha})+\sum_{j=1}^{\infty}\varepsilon_{j}(\lambda h^{\alpha})h^{j}\label{eq:}\end{equation}

\begin{center}
\includegraphics[scale=0.5]{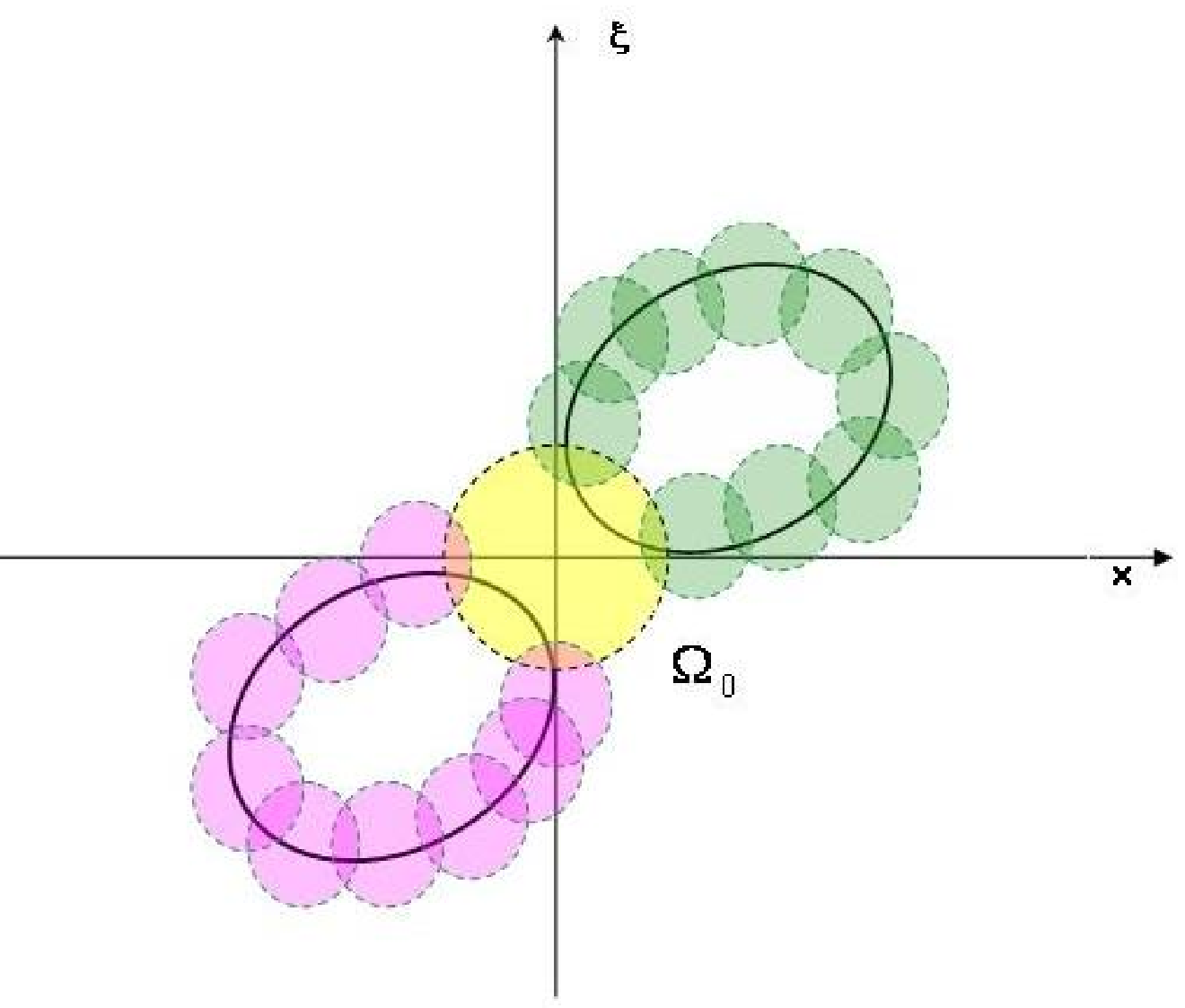}
\par\end{center}

\begin{center}
\textit{Fig. 5: Ouvert $\Omega_{0}$ où la forme normale est valide.}
\par\end{center}

$\vspace{0.25cm}$

Grâce à ce théorème on a un lien très simple entres les vecteurs propres
de ${\displaystyle \left(P_{h}-EI_{d}\right)}$ et ceux de ${\displaystyle \widehat{x\xi}}-\varepsilon(E)I_{d}$;
en effet on voit facilement que :\[
{\displaystyle \left(P_{h}-EI_{d}\right)}u_{h}=O(h^{\infty})\Leftrightarrow\left({\displaystyle \widehat{x\xi}}-\varepsilon(E)I_{d}\right){\displaystyle U}^{-1}(u_{h})=O(h^{\infty}).\]
 Ainsi si on travaille sur l'ouvert $\Omega_{0}$ où la forme normale
est valide, on est amené à résoudre : \textit{$\left({\displaystyle \widehat{x\xi}}-\varepsilon(E)I_{d}\right)v_{h}=O(h^{\infty})$,
}ie, résoudre :\textit{ $xv_{h}'(x)+\left(\frac{1}{2}-i\frac{\varepsilon(E)}{h}\right)v_{h}(x)=O(h^{\infty}).$}
Alors, par simple intégration d'équation différentielle ordinaire
linéaire du premier ordre, les solutions exactes de \textit{$\left({\displaystyle \widehat{x\xi}}-\varepsilon(E)I_{d}\right)v_{h}=0$}
sont engendrées par les deux fonctions :\[
\varphi_{1}(x):=x_{+}^{-\frac{1}{2}+i\frac{\varepsilon}{h}}=\mathbf{1}_{\mathbb{R}_{+}^{*}}(x)e^{-\frac{1}{2}\ln(x)+i\frac{\varepsilon}{h}\ln(x)}\]
et\[
\varphi_{2}(x):=x_{-}^{-\frac{1}{2}+i\frac{\varepsilon}{h}}=\mathbf{1}_{\mathbb{R}_{-}^{*}}(x)e^{-\frac{1}{2}\ln(-x)+i\frac{\varepsilon}{h}\ln(-x)}.\]
Ensuite, l'idée est de construire deux autres solutions de $\left({\displaystyle \widehat{x\xi}}-\varepsilon(E)I_{d}\right)v_{h}=0$;
pour cela on utilise la $h$-transformée de Fourier définie par \textit{:\[
\mathcal{F}_{h}(f)(x):=\frac{1}{\sqrt{2\pi h}}{\displaystyle \int_{-\infty}^{+\infty}f(t)e^{-\frac{ixt}{h}}\, dt.}\]
}En effet, en utilisant les propriétés usuelles sur la dérivation
des $h$-transformées de Fourier on a la : 

\begin{prop}
En posant $\varphi_{1}^{*}(x):=x_{+}^{-\frac{1}{2}-i\frac{\varepsilon}{h}}$
et $\varphi_{2}^{*}(x):=x_{-}^{-\frac{1}{2}-i\frac{\varepsilon}{h}}$,
les fonctions $\varphi_{3}$ et $\varphi_{4}$ définies par \[
\varphi_{3}(\xi):=e^{-i\frac{\pi}{4}}\mathcal{F}_{h}\left(\varphi_{1}^{*}\right)(-\xi)\;\;\, et\,\;\;\varphi_{4}(\xi):=e^{-i\frac{\pi}{4}}\mathcal{F}_{h}\left(\varphi_{2}^{*}\right)(-\xi)\]
sont aussi solutions exactes de ${\displaystyle \widehat{x\xi}}-\varepsilon(E)I_{d}=0$.
\end{prop}
Maintenant si $u_{h}$ est solution de \textit{$\left({\displaystyle \widehat{x\xi}}-\varepsilon(E)I_{d}\right)u_{h}=b_{h}$}
où le second membre $b_{h}$ est un $O(h^{\infty})$, on peut, en
utilisant essentiellement la méthode de la variation de la constante,
voir \textbf{{[}8]}, montrer que nécessairement $\exists!x_{1},x_{2}\in\mathbb{C}_{h}^{2}$
tels que $u_{h}=x_{1}\varphi_{1}+x_{2}\varphi_{2}+O(h^{\infty})$,
en effet :

\begin{thm}
\textbf{{[}8]} L'espace des solutions microlocales de l'équation $\left({\displaystyle \widehat{P}}-EI_{d}\right)u_{h}=O(h^{\infty})$
dans $\Omega_{0}$ est un $\mathbb{C}_{h}$- module libre de rang
$2$. 
\end{thm}
En notant par $\mathcal{B}:=\{\varphi_{1},\varphi_{2}\}$ et $\mathcal{B}^{\prime}:=\{\varphi_{3},\varphi_{4}\}$
les deux bases de solutions, la matrice de passage $Q$ de la base
$\mathcal{B}^{\prime}$ à $\mathcal{B}$ est donnée par :

\begin{thm}
En notant $\varepsilon:=\varepsilon(E)$, la matrice de passage $Q$
s'écrit :\[
Q=\mathcal{E}\left(\begin{array}{cc}
1 & ie^{-\frac{\varepsilon}{h}}\\
ie^{-\frac{\varepsilon}{h}} & 1\end{array}\right)\]
où \[
\mathcal{E}:=\frac{\Gamma(\frac{1}{2}+i\frac{\varepsilon}{h})}{\sqrt{2\pi}}e^{\frac{\varepsilon}{h}(\frac{\pi}{2}+\ln(h))}=\frac{1}{\sqrt{1+e^{-2\pi i\frac{\varepsilon}{h}}}}e^{i\arg(\Gamma(\frac{1}{2}+i\frac{\varepsilon}{h}))+i\frac{\varepsilon}{h}\ln(h)}.\]

\end{thm}
\begin{proof}
Pour cela on a besoin du (voir annexe) :
\end{proof}
\begin{lem}
Pour tout $\lambda\in\mathbb{C}-\mathbb{Z}^{*}$ on a, au sens des
distributions, que :\[
\mathcal{F}_{h}\left(\left[x_{+}^{\lambda}\right]\right)(\xi)=\frac{i\Gamma(\lambda+1)}{\sqrt{2\pi}}h^{\lambda+\frac{1}{2}}\left[e^{i\frac{\pi}{2}\lambda}\xi_{-}^{-\lambda-1}-e^{-i\frac{\pi}{2}\lambda}\xi_{+}^{-\lambda-1}\right]\]
et\[
\mathcal{F}_{h}\left(\left[x_{-}^{\lambda}\right]\right)(\xi)=\frac{i\Gamma(\lambda+1)}{\sqrt{2\pi}}h^{\lambda+\frac{1}{2}}\left[e^{i\frac{\pi}{2}\lambda}\xi_{+}^{-\lambda-1}-e^{-i\frac{\pi}{2}\lambda}\xi_{-}^{-\lambda-1}\right]\]
$\Gamma$ désignant la fonction Gamma d'Euler usuelle.
\end{lem}
De ce lemme, on en déduit l'égalité suivante :\[
\mathcal{F}_{h}\left(\varphi_{1}\right)(\xi)=\mathcal{F}_{h}\left(\left[x_{+}^{-\frac{1}{2}+i\frac{\varepsilon}{h}}\right]\right)(\xi)=\frac{i\Gamma(\frac{1}{2}+i\frac{\varepsilon}{h})}{\sqrt{2\pi}}h^{i\frac{\varepsilon}{h}}\left[e^{-\frac{\pi}{2}\frac{\varepsilon}{h}}e^{-i\frac{\pi}{4}}\xi_{-}^{-\frac{1}{2}-i\frac{\varepsilon}{h}}-e^{i\frac{\pi}{4}}e^{i\frac{\pi}{2}\frac{\varepsilon}{h}}\xi_{+}^{-\frac{1}{2}-i\frac{\varepsilon}{h}}\right].\]
En appliquant à nouveau $\mathcal{F}_{h}$ on a :\[
\left(\mathcal{F}_{h}\circ\mathcal{F}_{h}\right)(\varphi_{1})(x)=\frac{i\Gamma(\frac{1}{2}+i\frac{\varepsilon}{h})}{\sqrt{2\pi}}h^{i\frac{\varepsilon}{h}}\left[e^{-\frac{\pi}{2}\frac{\varepsilon}{h}}e^{-i\frac{\pi}{4}}\mathcal{F}_{h}\left(\left[\xi_{-}^{-\frac{1}{2}-i\frac{\varepsilon}{h}}\right]\right)(x)-e^{i\frac{\pi}{4}}e^{\frac{\pi}{2}\frac{\varepsilon}{h}}\mathcal{F}_{h}\left(\left[\xi_{+}^{-\frac{1}{2}-i\frac{\varepsilon}{h}}\right]\right)(x)\right]\]
ie : \[
\varphi_{1}(-x)=\frac{i\Gamma(\frac{1}{2}+i\frac{\varepsilon}{h})}{\sqrt{2\pi}}h^{i\frac{\varepsilon}{h}}\left[e^{-\frac{\pi}{2}\frac{\varepsilon}{h}}e^{-i\frac{\pi}{4}}\mathcal{F}_{h}\left(\varphi_{2}^{*}\right)(x)-e^{i\frac{\pi}{4}}e^{\frac{\pi}{2}\frac{\varepsilon}{h}}\mathcal{F}_{h}\left(\varphi_{1}^{*}\right)(x)\right]\]
\[
=\frac{i\Gamma(\frac{1}{2}+i\frac{\varepsilon}{h})}{\sqrt{2\pi}}h^{i\frac{\varepsilon}{h}}\left[e^{-\frac{\pi}{2}\frac{\varepsilon}{h}}\varphi_{4}(-x)-e^{i\frac{\pi}{2}}e^{\frac{\pi}{2}\frac{\varepsilon}{h}}\varphi_{3}(-x)\right]\]
et donc :

\[
\varphi_{1}(x)=\frac{\Gamma(\frac{1}{2}+i\frac{\varepsilon}{h})e^{\frac{\pi}{2}\frac{\varepsilon}{h}}}{\sqrt{2\pi}}h^{i\frac{\varepsilon}{h}}\left[\varphi_{3}(x)+ie^{-\pi\frac{\varepsilon}{h}}\varphi_{4}(x)\right].\]
De même on montre que \[
\varphi_{2}(x)=\frac{\Gamma(\frac{1}{2}+i\frac{\varepsilon}{h})e^{\frac{\pi}{2}\frac{\varepsilon}{h}}}{\sqrt{2\pi}}h^{i\frac{\varepsilon}{h}}\left[\varphi_{4}(x)+ie^{-\pi\frac{\varepsilon}{h}}\varphi_{3}(x)\right].\]
Par conséquent:\[
Q=\mathcal{E}\left(\begin{array}{cc}
1 & ie^{-\frac{\varepsilon}{h}\pi}\\
ie^{-\frac{\varepsilon}{h}\pi} & 1\end{array}\right)\textrm{ avec }\mathcal{E}=\frac{\Gamma(\frac{1}{2}+i\frac{\varepsilon}{h})}{\sqrt{2\pi}}e^{\frac{\varepsilon}{h}(\frac{\pi}{2}+\ln(h))}.\]
Pour finir la démonstration, il reste juste à vérifier que :

\[
\frac{\Gamma(\frac{1}{2}+i\frac{\varepsilon}{h})}{\sqrt{2\pi}}e^{\frac{\varepsilon}{h}(\frac{\pi}{2}+\ln(h))}=\frac{1}{\sqrt{1+e^{-2\pi\frac{\varepsilon}{h}}}}e^{i\arg(\Gamma(\frac{1}{2}+i\frac{\varepsilon}{h}))+i\frac{\varepsilon}{h}\ln(h)}\]
en effet comme\[
\arg\left(\Gamma\left(\frac{1}{2}+i\frac{\varepsilon}{h}\right)\right)=-i\ln\left(\Gamma\left(\frac{1}{2}+i\frac{\varepsilon}{h}\right)\right)+i\ln\left(\left|\Gamma(\frac{1}{2}+i\frac{\varepsilon}{h})\right|\right)\]
on a donc\[
e^{i\arg(\Gamma(\frac{1}{2}+i\frac{\varepsilon}{h}))}=\frac{\Gamma(\frac{1}{2}+i\frac{\varepsilon}{h})}{\left|\Gamma(\frac{1}{2}+i\frac{\varepsilon}{h})\right|}.\]
Et comme \[
\left|\Gamma\left(\frac{1}{2}+i\frac{\varepsilon}{h}\right)\right|^{2}=\Gamma\left(\frac{1}{2}+i\frac{\varepsilon}{h}\right){\displaystyle \overline{\Gamma\left(\frac{1}{2}+i\frac{\varepsilon}{h}\right)}}=\Gamma\left(\frac{1}{2}+i\frac{\varepsilon}{h}\right)\Gamma\left(\frac{1}{2}-i\frac{\varepsilon}{h}\right)\]
en appliquant la formule des compléments on a : \[
\left|\Gamma\left(\frac{1}{2}+i\frac{\varepsilon}{h}\right)\right|^{2}=\frac{\pi}{\cos(\pi i\frac{\varepsilon}{h})}=\frac{\pi}{\cosh(\pi\frac{\varepsilon}{h})}\]
et donc \[
e^{i\arg(\Gamma(\frac{1}{2}+i\frac{\varepsilon}{h}))}=\frac{\Gamma(\frac{1}{2}+i\frac{\varepsilon}{h})}{\sqrt{2\pi}}\sqrt{e^{\pi\frac{\varepsilon}{h}}+e^{-\frac{\varepsilon}{h}\pi}}\]
ainsi \[
\frac{1}{\sqrt{1+e^{-2\pi\frac{\varepsilon}{h}}}}e^{i\arg(\Gamma(\frac{1}{2}+i\frac{\varepsilon}{h}))}=\frac{\Gamma(\frac{1}{2}+i\frac{\varepsilon}{h})}{\sqrt{2\pi}}e^{\frac{\varepsilon}{h}\frac{\pi}{2}}.\]
Ce qui montre le théorème 3.8. $\square$

Revenons maintenant à l'étude de ${\displaystyle \left(P_{h}-EI_{d}\right)}u_{h}=O(h^{\infty})$:
si $u_{h}$ est une solution globale non triviale, en se plaçant sur
l'ouvert $\Omega_{0}$ où la forme normale est valide, il existe alors
$x_{1,}x_{2},x_{3,}x_{4}\in\left(\mathbb{C}_{h}\right)^{4}$ tels
que $U^{-1}u_{h}=x_{1}\varphi_{1}+x_{2}\varphi_{2}=x_{3}\varphi_{3}+x_{4}\varphi_{4}.$
Ensuite en posant pour tout indice $j\in\left\{ 1,2,3,4\right\} ,$
$\phi_{j}:={\displaystyle U}\varphi_{j}$ , les deux familles $\mathcal{C}:=\{\phi_{1},\phi_{2}\}$
et $\mathcal{C}^{\prime}:=\{\phi_{3},\phi_{4}\}$ sont deux bases
de solutions de $\left(P_{h}-EI_{d}\right)u_{h}=O(h^{\infty})$ dans
l'ouvert $\Omega_{0}$. Donc, dans $\Omega_{0}$ on a $u_{h}=x_{1}\phi_{1}+x_{2}\phi_{2}=x_{3}\phi_{3}+x_{4}\phi_{4}.$
Et ainsi on a alors la relation matrice-vecteur suivante:

\textit{\begin{equation}
\left(\begin{array}{c}
x_{3}\\
x_{4}\end{array}\right)=Q\left(\begin{array}{c}
x_{1}\\
x_{2}\end{array}\right).\label{eq:}\end{equation}
}

\subsection*{Seconde étape : Étude globale}

Toutes les fibres \foreignlanguage{english}{$\Lambda_{E}:=p^{-1}(E)$
sont compactes, pour $E\neq0$, la fibre }$\Lambda_{0}:=p^{-1}(0)$
étant l'unique fibre singulière du feuilletage. L'ensemble $\Upsilon_{0}:=p^{-1}(0)-\Omega_{0}$
est une partie régulière de la fibre $\Lambda_{0}$, pour $E\in[-1,1]$
le faisceau des solutions microlocales de ${\displaystyle {\displaystyle \left(P_{h}-EI_{d}\right)}}u_{h}=O(h^{\infty})$
au dessus de $\Lambda_{E}$ est un fibré plat de dimension 1 (voir\textbf{
{[}33]},\textbf{{[}34]}). 

\begin{center}
\includegraphics[scale=0.5]{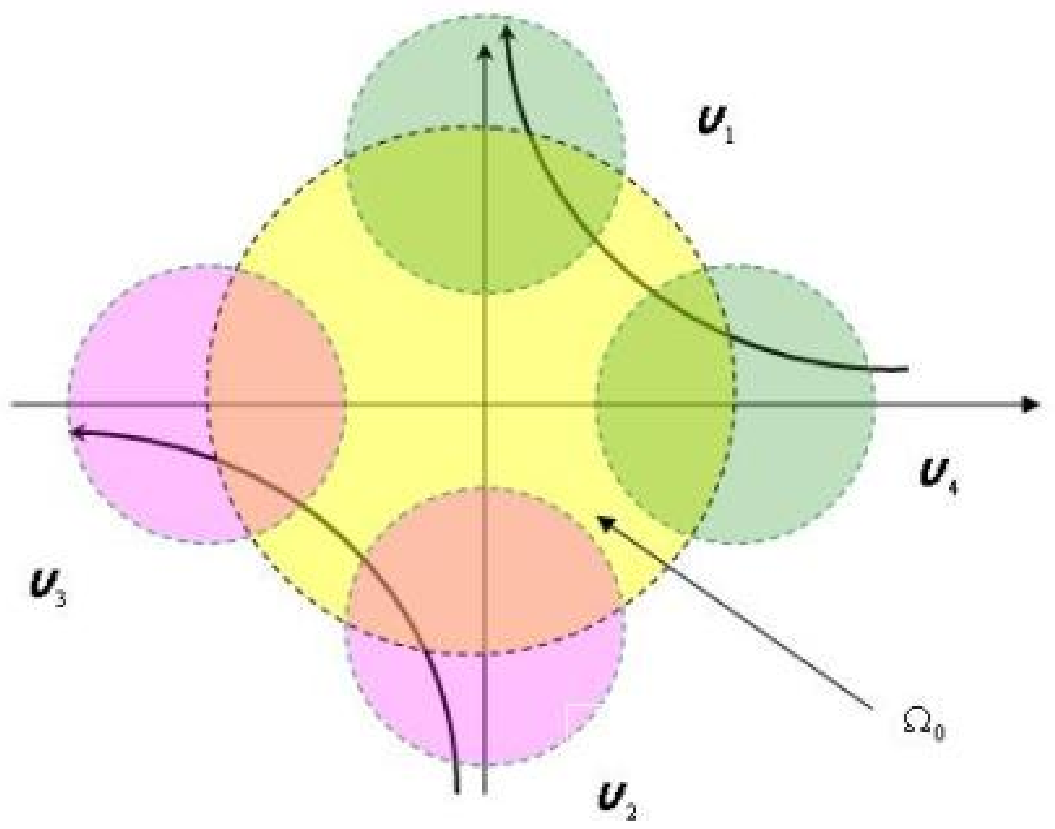}
\par\end{center}

\begin{center}
\textit{Fig. 6: Les ouverts $U_{1},U_{2},U_{3},U_{4}$ et $\Omega_{0}$.}
\par\end{center}

\vspace{+0.25cm}

La fonction $\phi_{1}$ appartient à $\mathcal{L}\left(P_{h},E,U_{1}\right)$
(ie : $\phi_{1}$ est une solution microlocale de \foreignlanguage{english}{${\displaystyle {\displaystyle \left({\displaystyle \widehat{P}}-EI_{d}\right)}}u_{h}=O(h^{\infty})$}
sur l'ouvert $U_{1}$) et la fonction $\phi_{4}$ appartient à $\mathcal{L}\left(P_{h},E,U_{4}\right)$,
donc d'après la proposition 2.9, il y a alors une unique façon de
prolonger (voir la sous-section 2.4) la solution $\phi_{1}$ le long
de la courbe en évitant la singularité pour arriver sur l'ouvert $U_{4}$;
la solution finale ${\displaystyle \widetilde{\phi_{1}}}$ diffère
alors de la solution $\phi_{4}$ par un facteur de phase (voir la
sous-section 2.4) : ${\displaystyle \widetilde{\phi_{1}}}=e^{iS^{+}(E)/h}\phi_{4}$
où la fonctions $S_{+}$ admet un développement asymptotique en puissance
de $h$ : $S^{+}(E)={\displaystyle \sum_{i=0}^{\infty}S_{j}^{+}(E)h^{j}}$
avec des coefficients \foreignlanguage{english}{$S_{j}^{+}$ qui sont}
$\mathcal{C}^{\infty}$ par rapport à la variable $E$. De la même
façon on a que ${\displaystyle \widetilde{\phi_{2}}}=e^{iS^{-}(E)/h}\phi_{3}$
avec aussi une fonctions $S_{-}$ ayant un développement asymptotique
en puissance de $h$ : $S^{-}(E)={\displaystyle \sum_{i=0}^{\infty}S_{j}^{-}(E)h^{j}}$
avec des coefficients \foreignlanguage{english}{$S_{j}^{-}$ qui sont}
$\mathcal{C}^{\infty}$ par rapport à la variable $E$. Ces deux séries
formelles $S^{+/-}$ sont appelées actions singulières. On posera
pour la suite\[
\theta_{+/-}(E):=\frac{S^{+/-}(E)}{h}.\]
A ce stade là, il ne reste plus qu'écrire les relations locales et
globales pour montrer le théorème : soit $u_{h}$ une solution globale
de ${\displaystyle {\displaystyle {\displaystyle \left({\displaystyle P_{h}}-EI_{d}\right)}}}u_{h}=O(h^{\infty})$,
telle que sur chacun des ouverts $U_{1},U_{2},U_{3},U_{4}$ (voir
Figure 6) on ait :\[
\forall j\in\left\{ 1,2,3,4\right\} ,\; u_{h\mid U_{j}}=x_{j}\phi_{j}\]
on a alors que : \[
\left(\begin{array}{c}
x_{3}\\
x_{4}\end{array}\right)=Q\left(\begin{array}{c}
x_{1}\\
x_{2}\end{array}\right)\textrm{ et }\left(\begin{array}{c}
x_{3}\\
x_{4}\end{array}\right)=\left(\begin{array}{cc}
0 & e^{i\theta_{-}(E)}\\
e^{i\theta_{+}(E)} & 0\end{array}\right)\left(\begin{array}{c}
x_{1}\\
x_{2}\end{array}\right).\]
Ainsi il existe une fonction $u_{h}=x_{1}\phi_{1}+x_{2}\phi_{2}=x_{3}\phi_{4}+x_{4}\phi_{4}$
solution globale non triviale de ${\displaystyle {\displaystyle {\displaystyle \left(P_{h}-EI_{d}\right)}}}u_{h}=O(h^{\infty})$
si et seulement si :\[
\det\left(Q-\left(\begin{array}{cc}
0 & e^{i\theta_{-}(E)}\\
e^{i\theta_{+}(E)} & 0\end{array}\right)\right)=0\]
\[
\Leftrightarrow\det\left(Q\left(\begin{array}{cc}
0 & e^{-i\theta_{+}(E)}\\
e^{-i\theta_{-}(E)} & 0\end{array}\right)-I_{2}\right)=0\]
\[
\Leftrightarrow1\in Spec(T(\mathcal{E}))\]
où on a posé\[
T(\mathcal{E}):=Q\left(\begin{array}{cc}
0 & e^{-i\theta_{+}(E)}\\
e^{-i\theta_{-}(E)} & 0\end{array}\right)=\mathcal{E}\left(\begin{array}{cc}
e^{-i\theta_{-}(E)}ie^{-\varepsilon(E)\pi/h} & e^{-i\theta_{+}(E)}\\
e^{-i\theta_{-}(E)} & e^{-i\theta_{+}(E)}ie^{-\varepsilon(E)\pi/h}\end{array}\right)\]
\[
=\mathcal{E}\left(\begin{array}{cc}
e^{-i\theta_{-}(E)}ie^{-\varepsilon(E)\pi/h} & e^{-i\theta_{+}(E)}\\
e^{-i\theta_{-}(E)} & e^{-i\theta_{+}(E)}ie^{-\varepsilon(E)\pi/h}\end{array}\right).\]
Et maintenant à ce stade là, pour conclure on utilise le lemme 1 de
\textbf{{[}8]}, rappelons le :

\begin{lem}
\textbf{{[}8]} Soit U une matrice unitaire de $M_{2}(\mathbb{C})$,
où $U=\left(\begin{array}{cc}
a & b\\
c & d\end{array}\right)$, tels que $U\neq\left(\begin{array}{cc}
0 & e^{i\theta_{1}}\\
e^{-i\theta_{2}} & 0\end{array}\right)$; alors\[
1\in Spec(U)\Leftrightarrow\left|a\right|\cos\left(\frac{\arg(da)}{2}-\arg(a)\right)=\cos\left(\frac{\arg(da)}{2}\right),\;\left|d\right|=\left|a\right|.\]

\end{lem}
En appliquant ce lemme à la matrice $T(\mathcal{E})$ on arrive bien
à :\[
1\in Spec(T(\mathcal{E}))\Leftrightarrow\]
\[
\frac{1}{\sqrt{1+e^{\frac{2\pi\varepsilon}{h}}}}\cos\left(\frac{\theta_{+}-\theta_{-}}{2}\right)=\cos\left(-\frac{\theta_{+}+\theta_{-}}{2}+\frac{\pi}{2}+\frac{\varepsilon}{h}\ln(h)+\arg\left(\Gamma\left(\frac{1}{2}+i\frac{\varepsilon}{h}\right)\right)\right).\]
Ce qui donne bien la formule proposée dans le théorème 3.2.

\begin{rem}
Dans \textbf{{[}10]}, Y. Colin de Verdière et B. Parisse montre que
dans le cas où $E=\lambda h$ avec $\lambda\in$$\left[-1,1\right]$,
les actions singulières peuvent s'écrire avec des invariants symplectiques
:\[
S_{0}^{+/-}(E)=A_{+/-}(E)+\varepsilon_{0}(E)\ln\left|\varepsilon_{0}(E)\right|-\varepsilon_{0}(E)\]
où $A_{+/-}(E):={\displaystyle \int}_{p=E,+/-}\xi\, dx$ est l'intégrale
d'action de la courbe $p^{-1}(E)$ du puit $+/-$.
\end{rem}

\subsection{Du singulier au régulier}

Toujours dans \textbf{{[}10]}, les auteurs examinent le lien entre
le cas régulier et singulier. Soient $E_{+},E_{-}\in\left[-1,1\right]^{2}$
tels que $0<E_{-}<E_{+}$ avec $E_{-}>\epsilon$ où $\epsilon$ est
un réel strictement positif indépendant de $h$.

\subsection*{Haut de spectre}

C'est le cas où $E\in I_{+}:=\left[E_{-},E_{+}\right]$, l'ensemble
$J_{+}:=p^{-1}\left(I_{+}\right)$ est alors un anneau topologique.
Pour tout $E\in I_{+},$ $\lim_{h\rightarrow0^{+}}\frac{\varepsilon}{h}=+\infty$
donc en utilisant la formule de Stirling, pour $h\rightarrow0$ on
a\[
\arg\left(\Gamma\left(\frac{1}{2}+i\frac{\varepsilon}{h}\right)\right)=\frac{\varepsilon}{h}\ln\left|\frac{\varepsilon}{h}\right|-\frac{\varepsilon}{h}+o(1)\]
d'où pour $h\rightarrow0$ : \[
-\frac{\theta_{+}+\theta_{-}}{2}+\frac{\pi}{2}+\frac{\varepsilon}{h}\ln(h)+\arg\left(\Gamma\left(\frac{1}{2}+i\frac{\varepsilon}{h}\right)\right)=-\frac{1}{h}A_{+/-}(E)+\frac{\pi}{2}+o(1).\]
D'autre part comme $\lim_{h\rightarrow0^{+}}\frac{1}{\sqrt{1+e^{\frac{2\pi\varepsilon}{h}}}}=0$;
l'asymptotique de la formule (3.1) est :\[
0=\cos\left(-\frac{A(E)}{2h}+\frac{\pi}{2}+o(1)\right)\]
où $A(E)=A_{+}(E)+A_{-}(E)$, ce qui donne bien les règles de Bohr-Sommerfeld
régulière pour un puits : $\frac{1}{h}A(E)+o(1)\in\pi\mathbb{Z}$.

\subsection*{Bas de spectre}

C'est le cas où $E\in I_{-}:=\left[-E_{+},-E_{-}\right]$, l'ensemble
$J_{-}:=p^{-1}\left(I_{-}\right)$ est alors la réunion de deux anneaux
topologique. Pour tout $E\in I_{-},$ $\lim_{h\rightarrow0^{+}}\frac{\varepsilon}{h}=-\infty$
donc toujours avec la formule de Stirling, pour $h\rightarrow0$ :\[
\arg\left(\Gamma\left(\frac{1}{2}+i\frac{\varepsilon}{h}\right)\right)=\frac{\varepsilon}{h}\ln\left|\frac{\varepsilon}{h}\right|-\frac{\varepsilon}{h}+o(1)\]
d'où pour $h\rightarrow0$ : \[
-\frac{\theta_{+}+\theta_{-}}{2}+\frac{\pi}{2}+\frac{\varepsilon}{h}\ln(h)+\arg\left(\Gamma\left(\frac{1}{2}+i\frac{\varepsilon}{h}\right)\right)=-\frac{1}{h}A_{+/-}(E)+\frac{\pi}{2}+o(1).\]
D'autre part comme $\lim_{h\rightarrow0^{+}}\frac{1}{\sqrt{1+e^{\frac{2\pi\varepsilon}{h}}}}=1$;
l'asymptotique de la formule (3.1) est :\[
\cos\left(\frac{A_{+}(E)-A_{-}(E)}{2h}+O(1)\right)=\cos\left(-\frac{A_{+}(E)+A_{-}(E)}{2h}+\frac{\pi}{2}+o(1)\right)\]
ce qui implique donc \[
\left\{ \begin{array}{cc}
\frac{A_{+}(E)}{h}+\frac{\pi}{2}+O(1)\in2\pi\mathbb{\mathbb{Z}}\\
\textrm{et}\\
\frac{A_{-}(E)}{h}+\frac{\pi}{2}+O(1)\in2\pi\mathbb{\mathbb{Z}}.\end{array}\right.\]
Ce sont bien les règles de Bohr-Sommerfeld régulières pour les deux
puits.

\section{La forme du spectre autour de la singularité}

\subsection{Introduction et résultats}

On va dans cette partie utiliser la formule du théorème 3.1 pour en
déduire des informations sur le spectre semi-classique autour de l'origine
de l'opérateur :\[
P_{h}:=-\frac{h^{2}}{2}\Delta+V.\]
Précisément on va démontrer le principal théorème de cette article
:

\begin{thm}
Le spectre semi-classique de l'opérateur ${\displaystyle P_{h}}$
sur le compact $\left[-\sqrt{h},\sqrt{h}\right]$ s'écrit comme la
réunion disjointe \[
\left(\alpha_{k}(h)\right)_{k\in I_{h}}\bigsqcup\left(\beta_{l}(h)\right)_{l\in J_{h}}\]
de deux familles $\left(\alpha_{k}(h)\right)_{k}$ et $\left(\beta_{l}(h)\right)_{l}$
s'écrivant $\alpha_{k}(h):=\sqrt{h}\mathcal{A}_{h}(2\pi k),\,\beta_{l}(h):=\sqrt{h}\mathcal{B}_{h}(2\pi l)$,
les fonctions $\mathcal{A}_{h}$ et $\mathcal{B}_{h}$ étant de classe
$\mathcal{C}^{\infty}$. De plus les familles $\left(\alpha_{k}(h)\right)_{k}$
et $\left(\beta_{l}(h)\right)_{l}$ sont strictement décroissantes
et en quinconce :\[
\beta_{k+1}(h)<\alpha_{k}(h)<\beta_{k}(h)<\alpha_{k-1}(h).\]
En outre l'interstice spectral est de l'ordre de $O(h/\left|\ln(h)\right|)$
: il existent $C,C^{\prime}$ deux constantes réelles strictement
positives telles que :\[
C\frac{h}{\left|\ln(h)\right|}\leq\left|\alpha_{k+1}(h)-\alpha_{k}(h)\right|\leq C^{\prime}\frac{h}{\left|\ln(h)\right|}.\]
De même pour la famille $\left(\beta_{l}(h)\right)_{l\in J_{h}}.$
\end{thm}
Qui a pour conséquence immédiate le :

\begin{cor}
Le nombre de valeurs propres de l'opérateur $P_{h}$ dans le compact
$\left[-\sqrt{h},\sqrt{h}\right]$ est de l'ordre de $\left|\ln(h)\right|/\sqrt{h}$.
\end{cor}
Avant de démontrer le théorème 4.1 on va interpréter le terme en $\ln(h)$.

\subsection{Interprétation géométrique}

Le terme $\left|\ln(h)\right|$ est la signature de la singularité
hyperbolique: en effet géométriquement il correspond au temps de parcours
du flot classique avec un point initial situé a une distance $\sqrt{h}$
de l'origine, en effet :

\begin{thm}
Soit $m_{h}\in T^{*}\mathbb{R}$ de coordonnés $\left(\sqrt{h},0\right)$
dans le repère $\left(0,x,\xi\right).$ Alors le flot hamiltonien
associé à $p$ et de point initial $m_{h}$ est périodique et sa période
$\tau_{h}$ vérifie pour $h\rightarrow0$ l'équivalent suivant :\[
\tau_{h}\sim\frac{\ln\left(h\right)}{K}\]
où $K$ est une constante réelle non nulle et indépendante de $h$.
\end{thm}
\begin{proof}
Sans perdre de généralités comme $V^{\prime\prime}(0)<0$ on peut
supposer que $-V^{\prime\prime}(0)=1$. Ensuite notons par $\Lambda_{h}=p^{-1}\left\{ p\left(m_{h}\right)\right\} $
l'unique fibre régulière contenant le point $m_{h}$ alors le flot
hamiltonien $\varphi_{t}\left(m_{h}\right)$ associé à $p$ et de
point initial $m_{h}$ est périodique et supporté sur la fibre $\Lambda_{h}$.
Pour estimer la période on va faire deux étapes : d'abord en se plaçant
autour de la singularité (en $0$) on peut utiliser une forme normale
classique pour estimer le temps de visite du flot dans un voisinage
de la singularité. Ensuite la seconde étape consiste à estimer le
temps de visite du flot en dehors de ce voisinage.

$\vspace{0.25cm}$

\textbf{Première étape. }Avant d'utiliser une forme normale on va
d'abord faire un changement de repère préliminaire : en faisant un
développement limité de la fonction $V$ autour de $0$ :\[
p(x,\xi)=\frac{\xi^{2}}{2}+V(x)=\frac{\xi^{2}}{2}+V(0)+V^{\prime}(0)+\frac{V^{\prime\prime}(0)}{2}x^{2}+O(x^{3})\]
\[
=\frac{\xi^{2}}{2}-\frac{x^{2}}{2}+o(x^{3})\]
donc sur un voisinage de $(0,0)$ nous avons que \[
p(x,\xi)=\left(\frac{\xi}{\sqrt{2}}-\frac{x}{\sqrt{2}}\right)\left(\frac{\xi}{\sqrt{2}}+\frac{x}{\sqrt{2}}\right)+o(x^{3}).\]
L'application :\[
\varphi\,:\,\left\{ \begin{array}{cc}
\mathbb{R}^{2}\rightarrow\mathbb{R}^{2}\\
\\(x,y)\mapsto\left(\frac{\xi}{\sqrt{2}}-\frac{x}{\sqrt{2}},\frac{\xi}{\sqrt{2}}+\frac{x}{\sqrt{2}}\right)\end{array}\right.\]
est un $\mathcal{C}^{1}$- difféomorphisme linéaire et son inverse
$\varphi^{-1}$ est égale à $\varphi.$ Ainsi dans les nouvelles variables
$\left(X,\Xi\right):=\varphi(x,\xi)$ on a\[
P(X,\Xi)=X\Xi+o\left(X^{3}\Xi^{3}\right)\]
et le point initial $m_{h}$ a pour nouvelles coordonnées $m_{h}=\left(\sqrt{\frac{h}{2}},\sqrt{\frac{h}{2}}\right)$.
Alors le théorème 2 de forme normale de Moser (voir\textbf{ {[}25]})
assure l'existence d'un ouvert $U$ de $\mathbb{R}^{2}$ contenant
l'origine, d'un symplectomorphisme $\psi\,:\, U\rightarrow\mathbb{R}^{2}$
et d'une fonction $F\,:\,\mathbb{R}\rightarrow\mathbb{R}$ de classe
$\mathcal{C}^{\infty}$ telle que pour tout $(X,\Xi)\in U$ on ait
\[
P(X,\Xi)=F(X\Xi).\]
 Ainsi les équations de Hamilton du flot sont alors :\[
\left\{ \begin{array}{cc}
\dot{X}=F^{\prime}\left(X\Xi\right)X\\
\\\dot{\Xi}=-F^{\prime}\left(X\Xi\right)\Xi.\end{array}\right.\]
Notons bien que $t\mapsto X(t)\Xi(t)$ est constante, ainsi pour tout
$t\geq0$ on a l'égalité $X(t)\Xi(t)=X(0)\Xi(0)=\frac{h}{2}$. En
posant $C_{h}=F^{\prime}\left(\frac{h}{2}\right)$ nous avons donc
que pour tout $t\geq0$ :\[
\left\{ \begin{array}{cc}
X(t)=\sqrt{\frac{h}{2}}e^{C_{h}t}\\
\\\Xi(t)=\sqrt{\frac{h}{2}}e^{-C_{h}t}.\end{array}\right.\]
Or comme $U$ est un ouvert non vide contenant $0$, il existe une
constante $A>0$ telle que la boule $B_{\infty}(0,A)$ (pour la distance
infinie de $\mathbb{R}^{2}$) de centre $0$ et de rayon $A$ soit
incluse dans $U$. On va calculer le temps $\tau_{1}(h)$ pour que
le flot hamiltonien partant du point $m_{h}=\left(X(0),\Xi(0)\right)=\left(\sqrt{\frac{h}{2}},\sqrt{\frac{h}{2}}\right)$
sorte de la boule carré $B_{\infty}(0,A)$: il faut donc trouver $t$
tel que $\Xi(t)=A$. On a alors immédiatement que :\[
\tau_{1}(h)=\frac{1}{C_{h}}\ln\left(\sqrt{\frac{h}{2}}\right)-\frac{1}{C_{h}}\ln(A)\]
\[
=\frac{1}{2C_{h}}\ln\left(h\right)-\frac{1}{2C_{h}}\ln(2)-\frac{1}{C_{h}}\ln(A).\]
Ainsi sur une période complète du flot hamiltonien partant du point
$m_{h}$ le temps total de parcours du flot dans la boule $B_{\infty}(0,A)$
est $2\tau_{1}(h)$.

$\vspace{0.25cm}$

\textbf{Seconde étape} \textbf{: }Il reste donc a estimer le temps
de parcours du flot en dehors de la boule $B_{\infty}(0,A)$. En fait,
on va montrer que ce temps est négligeable par rapport à $\tau_{1}(h).$
Considérons alors le point $a=(0,A)$ et comme l'unique fibre $\Lambda_{A}=p^{-1}\left\{ p(a)\right\} $
qui contienne ce point $a$ ne contient pas de singularité en dehors
de la boule $B_{\infty}(0,A)$, le flot hamiltonien de point initial
$a$ va nécessairement revenir en temps fini dans la boule $B_{\infty}(0,A)$,
on peut alors considérer le réel :\[
t^{*}:=\inf\left\{ t>0/\varphi_{t}(a)\in\overline{B_{\infty}(0,A)}\right\} \]
 et poser $b:=\varphi_{t^{*}}(a)$. Notons aussi par $T_{a}$ l'hyperplan
transverse au flot $\left(\varphi_{t}(a)\right)_{t\geq0}$ au point
$a$, et par $T_{b}$ l'hyperplan transverse au flot $\left(\varphi_{t}(a)\right)_{t\geq0}$
au point $b$. 

Comme le flot hamiltonien est associé au champs de vecteur $\mathcal{C}^{\infty}$:\[
\chi_{P}=\left(\begin{array}{c}
\Xi\\
V^{\prime}(X)\end{array}\right)\]
qui ne s'annule pas en $a$ et en $b$, par un théorème classique
de calcul différentiel de type application de Poincaré (voir par exemple
\textbf{{[}24]}) il existe $\Omega_{a}$ un voisinage ouvert de $a$
dans le plan $T_{a}$, une fonction $\theta$ de $\Omega_{a}$ dans
$\mathbb{R}$ de classe $\mathcal{C}^{\infty}$ telle que $\theta(a)=0$
avec les propriétés suivantes :

1) pour tout $x\in\Omega_{a}$ on a $\varphi_{t^{*}+\theta(x)}(x)\in T_{b}$;

2) l'application $x\mapsto\varphi_{t^{*}+\theta(x)}(x)$ est un difféomorphisme
local de $\Omega_{a}$ dans $\Omega_{b}$ un voisinage ouvert de $b$
dans le plan $T_{b}.$ 

Autrement dit, partant d'un point voisin de $a$ sur l'hyperplan $T_{a}$
le flot rencontre l'autre hyperplan $T_{b}$ en un temps voisin de
$t^{*}$ qui est une fonction différentiable du point de départ.

Donc en particulier comme $\Omega_{a}$ est un voisinage ouvert de
$a$ dans $T_{a}\simeq\mathbb{R}$, par compacité locale il existe
$K_{a}$ un compact de $\mathbb{R}$ tel que $a\in K_{a}\subset\Omega_{a}$
avec $K_{a}\neq\left\{ a\right\} $ et donc évidemment pour tout $x\in K_{a}$
on a $\left|\theta(x)\right|\leq{\displaystyle \sup_{x\in K_{a}}\left|\theta(x)\right|}$.
Ainsi, comme:\[
\varphi_{\tau_{1}(h)}(m_{h})=\left(\frac{h}{2A},A\right)\]
 pour $h$ assez petit on a que $\varphi_{\tau_{1}(h)}(m_{h})\in K_{a}\times\left\{ A\right\} $
et donc :\[
\left|\theta\left(\varphi_{\tau_{1}(h)}(m_{h})\right)\right|\leq\sup_{x\in K_{a}}\left|\theta(x)\right|\]
 d'où au final la période $\tau_{h}$ est égale à $\tau(h)=\tau_{1}(h)+\theta\left(\varphi_{\tau_{1}(h)}(m_{h})\right)$.

Au final :\[
\tau(h)=\frac{1}{2C_{h}}\ln\left(h\right)-\frac{1}{2C_{h}}\ln(2)-\frac{1}{C_{h}}\ln(A)+\theta\left(\varphi_{\tau_{1}(h)}(m_{h})\right)\]
donc pour $h\rightarrow0$ on a l'équivalent suivant :\[
\tau(h)\sim\frac{\ln\left(h\right)}{2F^{\prime}(\frac{h}{2})}\]
et comme $F^{\prime}\left(\frac{h}{2}\right)=F^{\prime}(0)+F^{\prime\prime}(0)\frac{h}{2}+o(h^{2})$
avec $F^{\prime}(0)\neq0$, d'où pour $h\rightarrow0$ l'équivalent
$\tau(h)\sim\frac{\ln\left(h\right)}{2F^{\prime}(0)}$.
\end{proof}

\subsection{Démonstration du théorème 4.1}

\subsection*{Stratégie de la preuve}

La formule du théorème de Colin de Verdière-Parisse (théorème 3.2)
est une équation fonctionnelle implicite, on va inverser (au sens
bijectif) cette fonction de manière à pouvoir expliciter les valeurs
propres. Pour cela on va utiliser ce théorème avec $E=\lambda h^{\alpha}$
où $\lambda\in[-1,1]$ et $\alpha\geq0$ . Par la suite on va voir
que si l'on choisit $\alpha\in\left[\frac{1}{2},1\right[$ de sorte
qu'on ai l'inclusion évidente $\left[-h^{\alpha},h^{\alpha}\right]\subseteq\left[-\sqrt{h},\sqrt{h}\right]$,
on peut montrer assez facilement le théorème 4.1 avec des techniques
d'analyse réelle basiques. Afin de comprendre pourquoi on suppose
$\alpha\in\left[\frac{1}{2},1\right[,$ plutôt qu'écrire la preuve
directement avec $\alpha=\frac{1}{2}$ on écrira toute la preuve avec
$\alpha\in\left[\frac{1}{2},1\right[$ (voir aussi la partie 4.4.).

\subsection*{Prologue}

On va commencer par des notations : pour alléger l'écriture on définit
sur le compact $\left[-1,1\right]$ les fonctions :\textit{\[
F_{h}(E):=-\frac{\theta_{+}(E)+\theta_{-}(E)}{2}+\frac{\pi}{2}+\frac{\varepsilon(E)}{h}\ln(h)+\arg\left(\Gamma\left(\frac{1}{2}+i\frac{\varepsilon(E)}{h}\right)\right)\]
}et\textit{\[
f_{h}(\lambda):=F_{h}(\lambda h^{\alpha})=-\frac{\theta_{+}(\lambda h^{\alpha})+\theta_{-}(\lambda h^{\alpha})}{2}+\frac{\pi}{2}+\frac{\varepsilon(\lambda h^{\alpha})}{h}\ln(h)+\arg\left(\Gamma\left(\frac{1}{2}+i\frac{\varepsilon(\lambda h^{\alpha})}{h}\right)\right)\]
}puis \textit{\[
G_{h}(E):=\frac{\theta_{+}(E)+\theta_{-}(E)}{2}\]
}et\textit{\[
g_{h}(\lambda):=G_{h}(\lambda h^{\alpha})=\frac{\theta_{+}(\lambda h^{\alpha})+\theta_{-}(\lambda h^{\alpha})}{2}.\]
}Pour finir avec les notations, sur le compact $\left[-1,1\right]$,
on définit les deux fonctions $\mathcal{Y}_{h}$ et $\mathcal{Z}_{h}$
par\[
\mathcal{Y}_{h}(\lambda):=f_{h}(\lambda)-\arccos\left(\frac{\cos\left(g_{h}(\lambda)\right)}{\sqrt{1+\exp\left(2\pi\varepsilon(\lambda h^{\alpha})/h\right)}}\right)\]
et\[
\mathcal{Z}_{h}(\lambda):=f_{h}(\lambda)+\arccos\left(\frac{\cos\left(g_{h}(\lambda)\right)}{\sqrt{1+\exp\left(2\pi\varepsilon(\lambda h^{\alpha})/h\right)}}\right).\]
Le théorème 3.2 affirme alors exactement que :\[
h^{\alpha}\lambda\in\Sigma_{h}({\displaystyle P_{h}},\left[-h^{\alpha},h^{\alpha}\right])\Leftrightarrow\frac{\cos\left(g_{h}(\lambda)\right)}{\sqrt{1+e^{2\pi\frac{\varepsilon}{h}}}}=\cos\left(f_{h}(\lambda)\right)\]
\[
\Leftrightarrow\left|\begin{array}{ccc}
f_{h}(\lambda)\equiv\arccos\left(\frac{\cos\left(g_{h}(\lambda)\right)}{\sqrt{1+\exp\left(2\pi\varepsilon(\lambda h^{\alpha})/h\right)}}\right)\;\left[2\pi\right]\\
\textrm{ou}\\
f_{h}(\lambda)\equiv-\arccos\left(\frac{\cos\left(g_{h}(\lambda)\right)}{\sqrt{1+\exp\left(2\pi\varepsilon(\lambda h^{\alpha})/h\right)}}\right)\;\left[2\pi\right]\end{array}\right.\]
\[
\Leftrightarrow\left|\begin{array}{ccc}
\mathcal{Y}_{h}(\lambda)\in2\pi\mathbb{Z}\\
\textrm{ou}\\
\mathcal{Z}_{h}(\lambda)\in2\pi\mathbb{Z}.\end{array}\right.\]
L'idée pour expliciter le spectre est d'inverser les fonctions $\mathcal{Y}_{h}$
et $\mathcal{Z}_{h}$ pour avoir une formule explicite. On va d'abord
montrer que :

\begin{prop}
Pour $h$ assez petit, la fonction $\mathcal{Y}_{h}$ (resp. la fonction
$\mathcal{Z}_{h}$) réalise une bijection strictement décroissante
de $\left[-1,1\right]$ sur $\mathcal{Y}_{h}\left(\left[-1,1\right]\right)$
(resp. sur $\mathcal{Z}_{h}\left(\left[-1,1\right]\right)$). En outre,
on a uniformément sur $\left[-1,1\right]$ que \[
\frac{h^{\alpha-1}\ln(h)}{\sqrt{-V^{''}(0)}}+O(h^{\alpha-1})\leq\mathcal{Y}_{h}^{\prime}(\lambda)\leq\frac{\alpha h^{\alpha-1}\ln(h)}{\sqrt{-V^{''}(0)}}+O(h^{\alpha-1}).\]
De même pour la fonction $\mathcal{Z}_{h}$.
\end{prop}
\begin{proof}
Avec la définition de la fonction $\mathcal{Y}_{h}$, pour tout $\lambda\in[-1,1]$
on a :\[
\mathcal{Y}_{h}^{\prime}(\lambda)=f_{h}^{\prime}(\lambda)-\frac{\partial}{\partial\lambda}\left[\arccos\left(\frac{\cos\left(g_{h}(\lambda)\right)}{\sqrt{1+\exp\left(2\pi\varepsilon(\lambda h^{\alpha})/h\right)}}\right)\right]\]
\textit{\[
=-h^{\alpha}\frac{\theta_{+}^{\prime}(\lambda h^{\alpha})+\theta_{-}^{\prime}(\lambda h^{\alpha})}{2}+h^{\alpha-1}\varepsilon^{\prime}(\lambda h^{\alpha})\ln(h)\]
}\[
+\frac{\partial}{\partial\lambda}\left[\arg\left(\Gamma\left(\frac{1}{2}+i\frac{\varepsilon(\lambda h^{\alpha})}{h}\right)\right)\right]-\frac{\partial}{\partial\lambda}\left[\arccos\left(\frac{\cos\left(g_{h}(\lambda)\right)}{\sqrt{1+\exp\left(2\pi\varepsilon(\lambda h^{\alpha})/h\right)}}\right)\right].\]
On va estimer, un par un, les quatre éléments de cette somme. 

Comme la fonction $E\mapsto-\left(\Theta_{+}^{\prime}(E)+\Theta_{-}^{\prime}(E)\right)/2$
admet un développement asymptotique de $-1$ à $+\infty$, avec des
coefficients $\mathcal{C}^{\infty}$ par rapport à $E$, on a donc
que la fonction $\lambda\mapsto-h^{\alpha}\frac{\theta_{+}^{\prime}(\lambda h^{\alpha})+\theta_{-}^{\prime}(\lambda h^{\alpha})}{2}$
admet un développement asymptotique de $\alpha-1$ à $+\infty$, avec
des coefficients $\mathcal{C}^{\infty}$ par rapport à $\lambda$,
ainsi nous avons que\[
-h^{\alpha}\frac{\theta_{+}^{\prime}(\lambda h^{\alpha})+\theta_{-}^{\prime}(\lambda h^{\alpha})}{2}=O(h^{\alpha-1}).\]
Ensuite on va estimer le terme $\lambda\mapsto h^{\alpha-1}\varepsilon^{\prime}(\lambda h^{\alpha})\ln(h)$:
en utilisant le développement asymptotique de la fonction $\varepsilon$
et en le dérivant on a : \[
\varepsilon^{\prime}(\lambda h^{\alpha})=\sum_{j=0}^{\infty}\varepsilon_{j}^{\prime}(\lambda h^{\alpha})h^{j}=\varepsilon_{0}^{\prime}(\lambda h^{\alpha})+\sum_{j=1}^{\infty}\varepsilon_{j}^{\prime}(\lambda h^{\alpha})h^{j}\]
\[
=\varepsilon_{0}^{\prime}(0)+O(h^{\alpha})+\sum_{j=1}^{\infty}\varepsilon_{j}^{\prime}(\lambda h^{\alpha})h^{j}\]
\[
=\frac{1}{\sqrt{-V^{''}(0)}}+O(h^{\alpha})+\sum_{j=1}^{\infty}\varepsilon_{j}^{\prime}(\lambda h^{\alpha})h^{j}.\]
Par conséquent nous obtenons :\[
h^{\alpha-1}\varepsilon^{\prime}(\lambda h^{\alpha})\ln(h)=\frac{h^{\alpha-1}\ln(h)}{\sqrt{-V^{''}(0)}}+O(h^{2\alpha-1}\ln(h))+\sum_{j=1}^{\infty}\varepsilon_{j}^{\prime}(\lambda h^{\alpha})h^{j+\alpha-1}\ln(h).\]
 Estimons maintenant le terme $\lambda\mapsto\frac{\partial}{\partial\lambda}\left[\arccos\left(\frac{\cos\left(g_{h}(\lambda)\right)}{\sqrt{1+\exp\left(2\pi\varepsilon(\lambda h^{\alpha})/h\right)}}\right)\right]$:
par un simple calcul de dérivé on a pour tout $\lambda\in[-1,1]$
l'égalité :\[
\frac{\partial}{\partial\lambda}\left[\arccos\left(\frac{\cos\left(g_{h}(\lambda)\right)}{\sqrt{1+\exp\left(2\pi\varepsilon(\lambda h^{\alpha})/h\right)}}\right)\right]\]
\[
=\frac{\sin(g_{h}(\lambda))g_{h}^{\prime}(\lambda)\left[1+\exp\left(2\pi\varepsilon(\lambda h^{\alpha})/h\right)\right]+\pi h^{\alpha-1}\varepsilon^{\prime}(\lambda h^{\alpha})\cos(g_{h}(\lambda))\exp\left(2\pi\varepsilon(\lambda h^{\alpha})/h\right)}{\left(1+\exp\left(2\pi\varepsilon(\lambda h^{\alpha})/h\right)\right)\sqrt{1+\exp\left(2\pi\varepsilon(\lambda h^{\alpha})/h\right)-\cos^{2}\left(g_{h}(\lambda)\right)}}\]
\[
=\frac{\sin(g_{h}(\lambda))g_{h}^{\prime}(\lambda)}{\sqrt{1+\exp\left(2\pi\varepsilon(\lambda h^{\alpha})/h\right)-\cos^{2}\left(g_{h}(\lambda)\right)}}\]
\[
+\frac{\pi h^{\alpha-1}\varepsilon^{\prime}(\lambda h^{\alpha})\cos(g_{h}(\lambda))\exp\left(2\pi\varepsilon(\lambda h^{\alpha})/h\right)}{\left(1+\exp\left(2\pi\varepsilon(\lambda h^{\alpha})/h\right)\right)\sqrt{1+\exp\left(2\pi\varepsilon(\lambda h^{\alpha})/h\right)-\cos^{2}\left(g_{h}(\lambda)\right)}}.\]
Or pour tout $\lambda\in[-1,1]$, comme : $1+\exp\left(2\pi\varepsilon(\lambda h^{\alpha})/h\right)\geq1$
on a donc que : \[
\sqrt{1+\exp\left(2\pi\varepsilon(\lambda h^{\alpha})/h\right)-\cos^{2}\left(g_{h}(\lambda)\right)}\geq\sqrt{1-\cos^{2}\left(g_{h}(\lambda)\right)}=\left|\sin\left(g_{h}(\lambda)\right)\right|\]
d'où pour tout $\lambda\in[-1,1]$:\[
\left|\frac{\sin(g_{h}(\lambda))g_{h}^{\prime}(\lambda)}{\sqrt{1+\exp\left(2\pi\varepsilon(\lambda h^{\alpha})/h\right)-\cos^{2}\left(g_{h}(\lambda)\right)}}\right|\leq\left|g_{h}^{\prime}(\lambda)\right|\]
\[
=\left|\frac{\partial}{\partial\lambda}\left[\frac{\theta_{+}(\lambda h^{\alpha})+\theta_{-}(\lambda h^{\alpha})}{2}\right]\right|=h^{\alpha}\left|\frac{\theta_{+}^{\prime}(\lambda h^{\alpha})+\theta_{-}^{\prime}(\lambda h^{\alpha})}{2}\right|\]
\textit{\begin{eqnarray*}
=\frac{h^{\alpha}}{2}\left|\frac{S_{0,+}^{\prime}(\lambda h^{\alpha})+S_{0,-}^{\prime}(\lambda h^{\alpha})}{h}\right.\end{eqnarray*}
\[
+\left.\left(S_{1,+}^{\prime}(\lambda h^{\alpha})+S_{1,-}^{\prime}(\lambda h^{\alpha})\right)+\sum_{j=2}^{\infty}\left(S_{j,+}^{\prime}(\lambda h^{\alpha})+S_{j,-}^{\prime}(\lambda h^{\alpha})\right)h^{j-1}\right|\]
}\[
=O(h^{\alpha-1}).\]
Ensuite comme pour tout $\lambda\in[-1,1]$ : \[
\sqrt{1+\exp\left(2\pi\varepsilon(\lambda h^{\alpha})/h\right)-\cos^{2}\left(g_{h}(\lambda)\right)}\geq\exp\left(\pi\varepsilon(\lambda h^{\alpha})/h\right)\]
nous avons que pour tout $\lambda\in[-1,1]$ :\[
\left|\frac{\pi h^{\alpha-1}\varepsilon^{\prime}(\lambda h^{\alpha})\cos(g_{h}(\lambda))\exp\left(2\pi\varepsilon(\lambda h^{\alpha})/h\right)}{\left(1+\exp\left(2\pi\varepsilon(\lambda h^{\alpha})/h\right)\right)\sqrt{1+\exp\left(2\pi\varepsilon(\lambda h^{\alpha})/h\right)-\cos^{2}\left(g_{h}(\lambda)\right)}}\right|\]
\[
\leq\left|\pi h^{\alpha-1}\varepsilon^{\prime}(\lambda h^{\alpha})\right|\exp\left(\pi\varepsilon(\lambda h^{\alpha})/h\right)\]
avec\[
\left|\pi h^{\alpha-1}\varepsilon^{\prime}(\lambda h^{\alpha})\right|=\left|\frac{\pi h^{\alpha-1}}{\sqrt{-V^{''}(0)}}+O(h^{2\alpha-1})+\sum_{j=1}^{\infty}\pi\varepsilon_{j}^{\prime}(\lambda h^{\alpha})h^{j+\alpha-1}\right|\]
\[
=O(h^{\alpha-1})\]
et\begin{equation}
\exp\left(\pi\varepsilon(\lambda h^{\alpha})/h\right)=\exp\left(\frac{\pi\lambda h^{\alpha-1}}{\sqrt{-V^{''}(0)}}+O(h^{2\alpha-1})+\sum_{j=1}^{\infty}\pi\varepsilon_{j}(\lambda h^{\alpha})h^{j-1}\right)\label{eq:}\end{equation}
comme $\alpha\geq\frac{1}{2}$, on en déduit alors que pour tout $\lambda\in[-1,0]$,
$\exp\left(\pi\varepsilon(\lambda h^{\alpha})/h\right)=O(1)$, ainsi
pour pour tout $\lambda\in[-1,0]$ on a :\[
\left|\frac{\pi h^{\alpha-1}\varepsilon^{\prime}(\lambda h^{\alpha})\cos(g_{h}(\lambda))\exp\left(2\pi\varepsilon(\lambda h^{\alpha})/h\right)}{\left(1+\exp\left(2\pi\varepsilon(\lambda h^{\alpha})/h\right)\right)\sqrt{1+\exp\left(2\pi\varepsilon(\lambda h^{\alpha})/h\right)-\cos^{2}\left(g_{h}(\lambda)\right)}}\right|=O(h^{\alpha-1}).\]
D'autre part, pour tout $\lambda\in[-1,1]$ on a aussi \[
\left|\frac{\pi h^{\alpha-1}\varepsilon^{\prime}(\lambda h^{\alpha})\cos(g_{h}(\lambda))\exp\left(2\pi\varepsilon(\lambda h^{\alpha})/h\right)}{\underbrace{\left(1+\exp\left(2\pi\varepsilon(\lambda h^{\alpha})/h\right)\right)}_{\geq\exp\left(2\pi\varepsilon(\lambda h^{\alpha})/h\right)}\underbrace{\sqrt{1+\exp\left(2\pi\varepsilon(\lambda h^{\alpha})/h\right)-\cos^{2}\left(g_{h}(\lambda)\right)}}_{\geq\exp\left(\pi\varepsilon(\lambda h^{\alpha})/h\right)}}\right|\]
\[
\leq\frac{\pi h^{\alpha-1}\varepsilon^{\prime}(\lambda h^{\alpha})\exp\left(2\pi\varepsilon(\lambda h^{\alpha})/h\right)}{\exp\left(3\pi\varepsilon(\lambda h^{\alpha})/h\right)}\]
\[
=\underbrace{\pi h^{\alpha-1}\varepsilon^{\prime}(\lambda h^{\alpha})}_{=O(h^{\alpha-1})}\exp\left(-\pi\varepsilon(\lambda h^{\alpha})/h\right)\]
avec\begin{equation}
\exp\left(-\pi\varepsilon(\lambda h^{\alpha})/h\right)=\exp\left(-\frac{\pi\lambda h^{\alpha-1}}{\sqrt{-V^{''}(0)}}+O(h^{2\alpha-1})-\sum_{j=1}^{\infty}\pi\varepsilon_{j}(\lambda h^{\alpha})h^{j-1}\right)\label{eq:}\end{equation}
donc toujours comme $\alpha\geq\frac{1}{2}$, on en déduit que pour
tout $\lambda\in[0,1]$ on a $\exp\left(-\pi\varepsilon(\lambda h^{\alpha})/h\right)=O(1)$,
ainsi pour tout $\lambda\in[0,1]$ on obtient \[
\left|\frac{\pi h^{\alpha-1}\varepsilon^{\prime}(\lambda h^{\alpha})\cos(g_{h}(\lambda))\exp\left(2\pi\varepsilon(\lambda h^{\alpha})/h\right)}{\left(1+\exp\left(2\pi\varepsilon(\lambda h^{\alpha})/h\right)\right)\sqrt{1+\exp\left(2\pi\varepsilon(\lambda h^{\alpha})/h\right)-\cos^{2}\left(g_{h}(\lambda)\right)}}\right|=O(h^{\alpha-1}).\]
On vient donc de montrer que pour tout $\lambda\in[-1,1]$ \[
\frac{\partial}{\partial\lambda}\left[\arccos\left(\frac{\cos\left(g_{h}(\lambda)\right)}{\sqrt{1+\exp\left(2\pi\varepsilon(\lambda h^{\alpha})/h\right)}}\right)\right]=O(h^{\alpha-1}).\]
Ensuite, pour finir, on va calculer et estimer $\lambda\mapsto\frac{\partial}{\partial\lambda}\left[\arg\left(\Gamma\left(\frac{1}{2}+i\frac{\varepsilon(\lambda h^{\alpha})}{h}\right)\right)\right]$:
pour tout $\lambda\in[-1,1]$ on a :\[
\frac{\partial}{\partial\lambda}\left[\arg\left(\Gamma\left(\frac{1}{2}+i\frac{\varepsilon(\lambda h^{\alpha})}{h}\right)\right)\right]=\frac{\partial}{\partial\lambda}\left[\textrm{Im}\left(\ln\left(\Gamma\left(\frac{1}{2}+i\frac{\varepsilon(\lambda h^{\alpha})}{h}\right)\right)\right)\right]\]
\[
=\textrm{Im}\left[\frac{\partial}{\partial\lambda}\left(\ln\left(\Gamma\left(\frac{1}{2}+i\frac{\varepsilon(\lambda h^{\alpha})}{h}\right)\right)\right)\right]\]
\[
=\textrm{Im}\left[\frac{\Gamma^{\prime}\left(\frac{1}{2}+i\frac{\varepsilon(\lambda h^{\alpha})}{h}\right)ih^{\alpha-1}\varepsilon^{\prime}(\lambda h^{\alpha})}{\Gamma\left(\frac{1}{2}+i\frac{\varepsilon(\lambda h^{\alpha})}{h}\right)}\right]\]
\[
=h^{\alpha-1}\varepsilon^{\prime}(\lambda h^{\alpha})\textrm{Re}\left[\frac{\Gamma^{\prime}\left(\frac{1}{2}+i\frac{\varepsilon(\lambda h^{\alpha})}{h}\right)}{\Gamma\left(\frac{1}{2}+i\frac{\varepsilon(\lambda h^{\alpha})}{h}\right)}\right]\]
\[
=h^{\alpha-1}\varepsilon^{\prime}(\lambda h^{\alpha})\textrm{Re}\left(\Psi\left(\frac{1}{2}+i\frac{\varepsilon(\lambda h^{\alpha})}{h}\right)\right)\]
où $\Psi$ est la fonction di-Gamma définie sur $\mathbb{C}-\mathbb{Z}^{-}$
par $\Psi(z):=\frac{\Gamma^{\prime}(z)}{\Gamma(z)}$ (voir \textbf{{[}1]}).
Rappelons que pour tout $\lambda\in[-1,1]$,

\[
\frac{\varepsilon(\lambda h^{\alpha})}{h}=\frac{\lambda h^{\alpha-1}}{\sqrt{-V^{''}(0)}}+O(h^{2\alpha-1})+\sum_{j=1}^{\infty}\varepsilon_{j}(\lambda h^{\alpha})h^{j-1}.\]
Donc comme $x\mapsto\textrm{Re}\left(\Psi\left(\frac{1}{2}+ix\right)\right)$
est paire et strictement croissante sur $\mathbb{R}_{+},$ on en déduit
l'encadrement pour tout $\lambda\in[-1,1]$ :\[
\textrm{Re}\left(\Psi\left(\frac{1}{2}+iO(h^{2\alpha-1})\right)\right)\leq\textrm{Re}\left(\Psi\left(\frac{1}{2}+i\frac{\varepsilon(\lambda h^{\alpha})}{h}\right)\right)\leq\textrm{Re}\left(\Psi\left(\frac{1}{2}+\frac{ih^{\alpha-1}}{\sqrt{-V^{''}(0)}}+iO\left(h^{2\alpha-1}\right)\right)\right).\]
Alors d'une part, comme\textbf{ }$\alpha\geq\frac{1}{2}$ nous avons
que pour tout $\lambda\in[-1,1]$:\[
\textrm{Re}\left(\Psi\left(\frac{1}{2}+iO(h^{2\alpha-1})\right)\right)=O(1).\]
D'autre part, comme (voir \textbf{{[}1]}) pour $|y|\rightarrow+\infty$
\[
\textrm{Re}\left(\Psi\left(\frac{1}{2}+iy\right)\right)=\ln\left|y\right|+O\left(\frac{1}{y^{2}}\right)\]
on en déduit (car $\alpha<1)$ que:\[
\textrm{Re}\left(\Psi\left(\frac{1}{2}+\frac{ih^{\alpha-1}}{\sqrt{-V^{''}(0)}}+iO(h^{2\alpha-1})\right)\right)\]
\[
=\ln\left|\frac{h^{\alpha-1}}{\sqrt{-V^{''}(0)}}+O(h^{2\alpha-1})\right|+O\left(\frac{1}{\left(\frac{h^{\alpha-1}}{\sqrt{-V^{''}(0)}}+O(h^{2\alpha-1})\right)^{2}}\right).\]
Or comme \[
\frac{1}{\left(\frac{h^{\alpha-1}}{\sqrt{-V^{''}(0)}}+O(h^{2\alpha-1})\right)^{2}}=\frac{1}{\frac{h^{2\alpha-2}}{-V^{''}(0)}+O(h^{3\alpha-2})}\]

\[
=\frac{-V^{''}(0)h^{2-2\alpha}}{1+O(h^{\alpha})}=-V^{''}(0)h^{2-2\alpha}+O(h^{2-\alpha})=O(h^{2-2\alpha})\]
et que\[
\ln\left|\frac{h^{\alpha-1}}{\sqrt{-V^{''}(0)}}+O(h^{2\alpha-1})\right|=\ln\left|\frac{h^{\alpha-1}}{\sqrt{-V^{''}(0)}}\left(1+O(h^{\alpha})\right)\right|\]
\[
=\ln\left|\frac{h^{\alpha-1}}{\sqrt{-V^{''}(0)}}\right|+\ln\left|1+O(h^{\alpha})\right|\]
\[
=(\alpha-1)\ln\left|h\right|-\ln\left|\sqrt{-V^{''}(0)}\right|+O(h^{\alpha}),\]
on en déduit que 

\selectlanguage{english}%
\[
\textrm{Re}\left(\Psi\left(\frac{1}{2}+\frac{ih^{\alpha-1}}{\sqrt{-V^{''}(0)}}+iO(h^{2\alpha-1})\right)\right)=(\alpha-1)\ln\left|h\right|-\ln\left|\sqrt{-V^{''}(0)}\right|+O(h^{\alpha})+O(h^{2-2\alpha}).\]
\foreignlanguage{french}{Par conséquent, pour tout $\lambda\in[-1,1]$
nous avons l'encadrement :\[
m_{\alpha}(h)\leq\frac{\partial}{\partial\lambda}\left[\arg\left(\Gamma\left(\frac{1}{2}+i\frac{\varepsilon(\lambda h^{\alpha})}{h}\right)\right)\right]\leq M_{\alpha}(h).\]
Où on a posé:\[
m_{\alpha}(h):=h^{\alpha-1}\varepsilon^{\prime}(\lambda h^{\alpha})\textrm{Re}\left(\Psi\left(\frac{1}{2}+iO(h^{2\alpha-1})\right)\right)=O(h^{\alpha-1})\]
et\[
M_{\alpha}(h):=h^{\alpha-1}\varepsilon^{\prime}(\lambda h^{\alpha})\textrm{Re}\left(\Psi\left(\frac{1}{2}+\frac{ih^{\alpha-1}}{\sqrt{-V^{''}(0)}}+iO(h^{2\alpha-1})\right)\right)\]
\[
=\left[\frac{h^{\alpha-1}}{\sqrt{-V^{''}(0)}}+O(h^{2\alpha-1})\right]\left[(\alpha-1)\ln\left|h\right|-\ln\left|\sqrt{-V^{''}(0)}\right|+O(h^{\alpha})+O(h^{2-2\alpha})\right]\]
\[
=\frac{(\alpha-1)h^{\alpha-1}}{\sqrt{-V^{''}(0)}}\ln(h)+O\left(h^{\alpha-1}\right).\]
Ainsi au final, on en déduit que pour tout $\lambda\in[-1,1]$ :\[
\frac{h^{\alpha-1}\ln(h)}{\sqrt{-V^{''}(0)}}+O(h^{\alpha-1})\leq\mathcal{Y}_{h}^{\prime}(\lambda)\leq\frac{\alpha h^{\alpha-1}\ln(h)}{\sqrt{-V^{''}(0)}}+O(h^{\alpha-1}).\]
Ensuite pour $h$ assez petit on conclut que pour tout $\lambda\in[-1,1]$,
$\mathcal{Y}_{h}^{\prime}(\lambda)<0$ et donc la fonction $\mathcal{Y}_{h}$
est bien strictement décroissante sur le compact $[-1,1]$. De même
pour la fonction $\mathcal{Z}_{h}$. }
\end{proof}
\selectlanguage{french}%

\subsection*{Deux familles de valeurs propres}

Comme les fonctions $\mathcal{Y}_{h}$ et $\mathcal{Z}_{h}$ sont
toutes deux bijectives, considérons leurs bijections réciproques,
que l'on renote par:\[
\mathcal{A}_{h}:=\mathcal{Y}_{h}^{-1}\,:\,\mathcal{Y}_{h}\left([-1,1]\right)\rightarrow[-1,1]\;\;\textrm{et}\;\;\mathcal{B}_{h}:=\mathcal{Z}_{h}^{-1}\,:\,\mathcal{Z}_{h}\left([-1,1]\right)\rightarrow[-1,1].\]
Ainsi la condition nécessaire et suffisante des valeurs propres semi-classique
: \[
h^{\alpha}\lambda\in\Sigma_{h}(P_{h},\left[-h^{\alpha},h^{\alpha}\right])\]
\[
\Leftrightarrow\left|\begin{array}{ccc}
\mathcal{Y}_{h}(\lambda)\in2\pi\mathbb{Z} &  & \textrm{avec}\,\lambda\in[-1,1]\\
\textrm{ou}\\
\mathcal{Z}_{h}(\lambda)\in2\pi\mathbb{Z} &  & \textrm{avec}\,\lambda\in[-1,1]\end{array}\right.\]
\[
\Leftrightarrow\lambda\in\left({\displaystyle \bigcup_{k\in I_{h}}}\mathcal{A}_{h}(2\pi k)\right)\bigcup\left({\displaystyle \bigcup_{l\in J_{h}}}\mathcal{B}_{h}(2\pi l)\right)\]
où on à posé\[
I_{h}:=\left\{ k\in\mathbb{Z}/2\pi k\in\mathcal{Y}_{h}\left([-1,1]\right)\right\} =\frac{\mathcal{Y}_{h}\left([-1,1]\right)}{2\pi}\cap\mathbb{Z}\]
 et\[
J_{h}:=\left\{ l\in\mathbb{Z}/2\pi l\in\mathcal{Z}_{h}\left([-1,1]\right)\right\} =\frac{\mathcal{Z}_{h}\left([-1,1]\right)}{2\pi}\cap\mathbb{Z}.\]
En résumant nous avons alors la :

\begin{prop}
L'équation ${\displaystyle {\displaystyle {\displaystyle \left({\displaystyle P_{h}}-h^{\alpha}\lambda I_{d}\right)}}}u_{h}=O(h^{\infty})$
admet une solution $u_{h}\in L^{2}(\mathbb{R})$ non triviale avec
son microsupport $MS(u_{h})=p^{-1}\{0\}$ si et seulement si :\[
\lambda\in\left({\displaystyle \bigcup_{k\in I_{h}}}\mathcal{A}_{h}(2\pi k)\right)\bigcup\left({\displaystyle \bigcup_{k\in J_{h}}}\mathcal{B}_{h}(2\pi k)\right)\]
où $\mathcal{A}_{h}=\mathcal{Y}_{h}^{-1}$, $\mathcal{B}_{h}=\mathcal{Z}_{h}^{-1}$
et $I_{h}=\frac{\mathcal{Y}_{h}\left([-1,1]\right)}{2\pi}\cap\mathbb{Z}$,
$J_{h}=\frac{\mathcal{Z}_{h}\left([-1,1]\right)}{2\pi}\cap\mathbb{Z}$.
\end{prop}
Notons bien que les ensembles $I_{h}$ et $J_{h}$ ne sont pas vides,
en effet :

\begin{prop}
Pour $h$ assez petit, nous avons les encadrements suivants :\[
E\left[-\frac{\alpha h^{\alpha-1}\ln(h)}{\pi\sqrt{-V^{''}(0)}}+O(h^{\alpha-1})\right]\leq\textrm{Card}(I_{h})\leq E\left[-\frac{h^{\alpha-1}\ln(h)}{\pi\sqrt{-V^{''}(0)}}+O(h^{\alpha-1})\right]+1\]
où $E[x]$ désigne la partie entière de $x.$ On a le même encadrement
pour le cardinal de l'ensemble $J_{h}$.
\end{prop}
\begin{proof}
On va faire la preuve uniquement pour l'ensemble $I_{h}$. Comme la
fonction $\mathcal{Y}_{h}$ est strictement décroissante sur le compact
$\left[-1,1\right]$, le diamètre du compact $\mathcal{Y}_{h}\left(\left[-1,1\right]\right)$
est simplement donné par la relation:\[
\textrm{diam}\left(\mathcal{Y}_{h}\left(\left[-1,1\right]\right)\right)=\mathcal{Y}_{h}(-1)-\mathcal{Y}_{h}(1).\]
Par le théorème des accroissements finis il existe $\xi\in\left]-1,1\right[$
tels que : \[
\mathcal{Y}_{h}(-1)-\mathcal{Y}_{h}(1)=-2\mathcal{Y}_{h}^{\prime}(\xi)>0.\]
On obtient donc l'encadrement suivant :\[
-2\alpha\frac{h^{\alpha-1}\ln(h)}{\sqrt{-V^{''}(0)}}+O(h^{\alpha-1})\leq\textrm{diam}\left(\mathcal{Y}_{h}\left(\left[-1,1\right]\right)\right)\leq\frac{-2h^{\alpha-1}\ln(h)}{\sqrt{-V^{''}(0)}}+O(h^{\alpha-1}).\]
La suite de la preuve est alors directe. 
\end{proof}

\subsection*{Quinconce et interstice }

Comme on l'a vu, dans le compact $\left[-h^{\alpha},h^{\alpha}\right]$
(avec $\alpha\geq\frac{1}{2}$) le spectre semi-classique de l'opérateur
:\[
P_{h}=-\frac{h^{2}}{2}\frac{d^{2}}{dx^{2}}+V\]
est constitué de deux familles : d'abord la famille\[
\alpha_{k}(h):=h^{\alpha}\mathcal{A}_{h}(2\pi k),\, k\in I_{h}\]
 puis la famille\[
\beta_{l}(h):=h^{\alpha}\mathcal{B}_{h}(2\pi l),\, l\in J_{h}.\]
 Donnons les propriétés importantes de ces deux familles.

\begin{prop}
Pour h assez petit, les deux familles de réels $\left(\alpha_{k}(h)\right)_{k\in I_{h}}$
et $\left(\beta_{l}(h)\right)_{l\in J_{h}}$ sont strictement décroissantes.
\end{prop}
\begin{proof}
cela tiens juste du fait que les fonctions $\mathcal{Y}_{h}$ et $\mathcal{Z}_{h}$
sont $\mathcal{C}^{1}$ et strictement décroissantes, donc leurs bijections
réciproques le sont aussi.
\end{proof}
\begin{lem}
La famille\[
\left\{ \left(\alpha_{n}(h)\right)_{n\in I_{h}},\left(\beta_{l}(h)\right)_{l\in J_{h}}\right\} \]
 est une famille de réels deux à deux bien distincts.
\end{lem}
\begin{proof}
Les familles $\left\{ \alpha_{n}(h)\right\} _{n\in I_{h}}$ et \textbf{$\left\{ \beta_{l}(h)\right\} _{l\in J_{h}}$}
étant des familles de réels strictement décroissantes, il suffit juste
de vérifier que ces deux familles n'ont pas de valeur commune. Raisonnons
par l'absurde : supposons qu'il existent $(k,l)\in I_{h}\times J_{h}$
tels que $\alpha_{k}(h)=\beta_{l}(h)$, ie : $\mathcal{A}_{h}(2\pi k)=\mathcal{B}_{h}(2\pi l)$.
En notant par $\lambda$ cette valeur commune, c'est-à-dire :\[
\lambda:=\mathcal{A}_{h}(2\pi k)=\mathcal{B}_{h}(2\pi l)\]
puis en appliquant les fonctions $\mathcal{Y}_{h}$ et $\mathcal{Z}_{h}$
sur le réel $\lambda$, on a que \[
\mathcal{Y}_{h}(\lambda)=2\pi k\in2\pi\mathbb{Z}\;\;\textrm{et}\;\;\mathcal{Z}_{h}(\lambda)=2\pi l\in2\pi\mathbb{Z}\]
et par conséquent : \[
\mathcal{Y}_{h}(\lambda)-\mathcal{Z}_{h}(\lambda)\in2\pi\mathbb{Z}\]
donc par définition des fonctions $\mathcal{Y}_{h}$ et $\mathcal{Z}_{h}$
nous avons \[
-2\arccos\left(\frac{\cos\left(g_{h}(\lambda)\right)}{\sqrt{1+\exp\left(2\pi\varepsilon(\lambda h^{\alpha})/h\right)}}\right)\in2\pi\mathbb{Z}\]
d'où :\[
\arccos\left(\frac{\cos\left(g_{h}(\lambda)\right)}{\sqrt{1+\exp\left(2\pi\varepsilon(\lambda h^{\alpha})/h\right)}}\right)\in\pi\mathbb{Z}\]
ainsi nécessairement on a \[
\frac{\cos\left(g_{h}(\lambda)\right)}{\sqrt{1+\exp\left(2\pi\varepsilon(\lambda h^{\alpha})/h\right)}}\in\left\{ -1,1\right\} .\]
Ce qui implique finalement l'égalité : \[
\underbrace{\cos^{2}\left(g_{h}(\lambda)\right)}_{\leq1}=\underbrace{1+\exp\left(2\pi\varepsilon(\lambda h^{\alpha})/h\right)}_{>1}\]
qui est absurde, d'où le lemme proposé.
\end{proof}

\subsection*{Quinconce et interstice }

On va maintenant s'intéresser à comparer ces deux familles entre elles,
pour cela il faut prendre des indices appartenant à $I_{h}\cap J_{h}$.
On va donc d'abord s'assurer que $I_{h}\cap J_{h}\mathbb{\subset Z}$
est non vide. 

\begin{prop}
Pour h assez petit, nous avons\[
\textrm{E}\left[\frac{\alpha(\xi-1)h^{\alpha-1}\ln(h)}{\pi\sqrt{-V^{''}(0)}}+O(h^{\alpha-1})\right]\leq\textrm{Card}(I_{h}\cap J_{h})\leq\textrm{E}\left[\frac{(\xi-1)h^{\alpha-1}\ln(h)}{\pi\sqrt{-V^{''}(0)}}+O(h^{\alpha-1})\right]+1\]
où $\xi\in]-1,1[$.
\end{prop}
\begin{proof}
Écrivons juste la différence entre les fonctions $\mathcal{Y}_{h}$
et $\mathcal{Z}_{h}$ , pour tout $\lambda\in[-1,1]$ nous avons donc
que\[
\mathcal{Y}_{h}(\lambda)-\mathcal{Z}_{h}(\lambda)=\underbrace{-2\arccos\left(\frac{\cos\left(g_{h}(\lambda)\right)}{\sqrt{1+\exp\left(2\pi\varepsilon(\lambda h^{\alpha})/h\right)}}\right)}_{\in\left[-2\pi,0\right]}\]
donc en particulier\[
\mathcal{Y}_{h}(-1)-\mathcal{Z}_{h}(-1)<0\;\;\textrm{et}\;\;\mathcal{Y}_{h}(1)-\mathcal{Z}_{h}(1)<0\]
(pour le strict dans les inégalités, voir la démonstration du précédent
lemme).

Ensuite comme d'après la preuve de la proposition 4.6 on a l'encadrement
:\[
-2\alpha\frac{h^{\alpha-1}\ln(h)}{\sqrt{-V^{''}(0)}}+O(h^{\alpha-1})\leq\textrm{diam}(\mathcal{Y}_{h}([-1,1]))\leq\frac{-2h^{\alpha-1}\ln(h)}{\sqrt{-V^{''}(0)}}+O(h^{\alpha-1})\]
et en utilisant aussi que pour tout $\lambda\in[-1,1]$ \[
\left|\mathcal{Y}_{h}(\lambda)-\mathcal{Z}_{h}(\lambda)\right|\leq2\pi\]
on voit immédiatement que pour $h$ assez petit $\mathcal{Y}_{h}\left(\left[-1,1\right]\right)\cap\mathcal{Z}_{h}(\left[-1,1\right])\neq\emptyset$
; et on a même mieux, en effet comme: \[
\textrm{diam}\left(\mathcal{Y}_{h}\left(\left[-1,1\right]\right)\cap\mathcal{Z}_{h}\left(\left[-1,1\right]\right)\right)=\mathcal{Y}_{h}(-1)-\mathcal{Z}_{h}(1)\]
puis que\[
\mathcal{Z}_{h}(1)\leq\mathcal{Y}_{h}(-1)\leq\mathcal{Z}_{h}(-1)\]
par le théorème des valeurs intermédiaires il existe $\xi\in\left[-1,1\right]$
tels que\[
\mathcal{Y}_{h}(-1)=\mathcal{Z}_{h}(\xi)\]
par conséquent nous avons \[
\textrm{diam}\left(\mathcal{Y}_{h}\left(\left[-1,1\right]\right)\cap\mathcal{Z}_{h}\left(\left[-1,1\right]\right)\right)=\mathcal{Z}_{h}(\xi)-\mathcal{Z}_{h}(1)\]
\[
=\mathcal{Z}_{h}^{\prime}(\theta)(\xi-1)\]
où $\theta\in\left]\xi,1\right[$ est donné par le théorème des accroissements
finis, d'où au final :\[
\frac{\alpha h^{\alpha-1}\ln(h)(\xi-1)}{\sqrt{-V^{''}(0)}}+O(h^{\alpha-1})\leq\textrm{diam}\left(\mathcal{Y}_{h}\left(\left[-1,1\right]\right)\cap\mathcal{Z}_{h}\left(\left[-1,1\right]\right)\right)\leq\frac{h^{\alpha-1}\ln(h)(\xi-1)}{\sqrt{-V^{''}(0)}}+O(h^{\alpha-1})\]
et on en déduit alors la proposition.
\end{proof}
\begin{prop}
Pour h assez petit et pour tout $k\in\frac{\mathcal{Y}_{h}\left([-1,1]\right)\cap\mathcal{Z}_{h}\left([-1,1]\right)}{2\pi}\cap\mathbb{Z}$,
on a que \[
\alpha_{k}(h)<\beta_{k}(h).\]

\end{prop}
\begin{proof}
On sait déjà que pour tout $\lambda\in[-1,1]$ \[
\mathcal{Y}_{h}(\lambda)-\mathcal{Z}_{h}(\lambda)=\underbrace{-2\arccos\left(\frac{\cos\left(g_{h}(\lambda)\right)}{\sqrt{1+\exp\left(2\pi\varepsilon(\lambda h^{\alpha})/h\right)}}\right)}_{\in\left[-2\pi,0\right]}\leq0.\]
Le lemme 4.8 nous informe de plus que la précédente inégalité est
stricte : pour tout $\lambda\in[-1,1]$ \[
\mathcal{Y}_{h}(\lambda)<\mathcal{Z}_{h}(\lambda).\]
De là on déduit que pour tout $k\in\frac{\mathcal{Y}_{h}(\left[-1,1\right])\cap\mathcal{Z}_{h}(\left[-1,1\right])}{2\pi}\cap\mathbb{Z}$
\[
\mathcal{Y}_{h}\left(\mathcal{A}_{h}(2\pi k)\right)<\mathcal{Z}_{h}\left(\mathcal{A}_{h}(2\pi k)\right)\]
ie : \[
2\pi k<\mathcal{Z}_{h}\left(\mathcal{A}_{h}(2\pi k)\right).\]
Comme $2\pi k\in\mathcal{Z}_{h}\left([-1,1]\right)$ et que $\mathcal{B}_{h}\,:\,\mathcal{Z}_{h}\left([-1,1]\right)\rightarrow[-1,1]$
en appliquant la fonction $\mathcal{B}_{h}$ (qui est strictement
décroissante) sur la dernière inégalité on arrive a :\[
\mathcal{B}_{h}(2\pi k)>\mathcal{A}_{h}(2\pi k)\]
et donc \[
\alpha_{k}(h)<\beta_{k}(h).\]
Ce qui finit la preuve.
\end{proof}
Ensuite on a la :

\begin{prop}
Pour h assez petit tels et pour tout $k\in\frac{\mathcal{Y}_{h}([-1,1])\cap\mathcal{Z}_{h}([-1,1])}{2\pi}\cap\mathbb{Z}$
nous avons :\[
\beta_{k}(h)<\alpha_{k-1}(h).\]

\end{prop}
\begin{proof}
Pour tout $k\in\frac{\mathcal{Y}_{h}([-1,1])\cap\mathcal{Z}_{h}([-1,1])}{2\pi}\cap\mathbb{Z}$,
considérons les deux réels : \[
\theta_{k}:=\mathcal{B}_{h}(2\pi k)\in[-1,1]\]
et \[
\zeta_{k}:=\mathcal{A}_{h}(2\pi k)-\mathcal{A}_{h}(2\pi(k-1))<0.\]
Alors comme :\[
\mathcal{Y}_{h}(\theta_{k}+\zeta_{k})-\mathcal{Z}_{h}(\theta_{k})\]
\[
=f_{h}(\theta_{k}+\zeta_{k})-f_{h}(\theta_{k})-\arccos\left(\frac{\cos\left(g_{h}(\theta_{k}+\zeta_{k})\right)}{\sqrt{1+\exp\left(2\pi\varepsilon(\lambda h^{\alpha})/h\right)}}\right)-\arccos\left(\frac{\cos\left(g_{h}(\theta_{k})\right)}{\sqrt{1+\exp\left(2\pi\varepsilon(\lambda h^{\alpha})/h\right)}}\right)\]
\[
=f_{h}(\theta_{k}+\zeta_{k})-f_{h}(\theta_{k})+O(1)\]
\[
=\underbrace{f_{h}^{\prime}(\tau_{k})\zeta_{k}+O(1)}_{>0\;(\textrm{car}\,\zeta_{k}<0)}\]
où $\tau_{k}$ est donné par le théorème des accroissement finis,
on a que :

\[
\mathcal{Y}_{h}(\theta_{k}+\zeta_{k})>\mathcal{Z}_{h}(\theta_{k})\]
ie : \[
\mathcal{Y}_{h}(\mathcal{B}_{h}(2\pi k)+\zeta_{k})>2\pi k.\]
D'où en appliquant la fonction $\mathcal{A}_{h}$ (qui est strictement
décroissante) nous obtenons alors:

\[
\mathcal{B}_{h}(2\pi k)+\zeta_{k}<\mathcal{A}_{h}(2\pi k)\]
ie :

\[
\mathcal{B}_{h}(2\pi k)+\mathcal{A}_{h}(2\pi k)-\mathcal{A}_{h}(2\pi(k-1))<\mathcal{A}_{h}(2\pi k)\]
soit encore 

\[
\beta_{k}(h)<\alpha_{k-1}(h).\]
Ce qui montre l'inégalité proposée dans l'énoncé.
\end{proof}
Pour finir, estimons la distance entre les valeurs propres :

\begin{prop}
Il existent $C$ et $C^{\prime}$ deux nombres réels strictement positifs
et indépendant de $h$ tels que :\[
C\frac{h}{\left|\ln(h)\right|}\leq\left|\alpha_{k+1}(h)-\alpha_{k}(h)\right|\leq C^{\prime}\frac{h}{\left|\ln(h)\right|}.\]
De même pour la distance $\left|\beta_{l+1}(h)-\beta_{l}(h)\right|.$
\end{prop}
\begin{proof}
Or pour tout indice $k\in\frac{\mathcal{Y}_{h}([-1,1])\cap\mathcal{Z}_{h}([-1,1])}{2\pi}\cap\mathbb{Z}$,
on a : \[
\left|\alpha_{k+1}(h)-\alpha_{k}(h)\right|=h^{\alpha}\left|\mathcal{A}_{h}\left(2\pi(k+1)\right)-\mathcal{A}_{h}\left(2\pi k\right)\right|\]
\[
=h^{\alpha}\left|\mathcal{A}_{h}^{\prime}(\xi_{k})2\pi\right|\]
 où $\xi_{k}\in\left]k,k+1\right[$ est donné par le théorème des
accroissements finis. Il reste alors à écrire simplement que :\[
\left|\mathcal{A}_{h}^{\prime}(\xi_{k})\right|=\left|\frac{1}{\mathcal{Y}_{h}^{\prime}(\mathcal{A}_{h}(\xi_{k}))}\right|\]
pour avoir l'encadrement suivant :\[
\frac{2\pi h^{\alpha}}{\frac{h^{\alpha-1}\left|\ln(h)\right|}{\sqrt{-V^{''}(0)}}+O(h^{\alpha-1})}\leq\left|\alpha_{k+1}(h)-\alpha_{k}(h)\right|\leq\frac{2\pi h^{\alpha}}{\frac{\alpha h^{\alpha-1}\left|\ln(h)\right|}{\sqrt{-V^{''}(0)}}+O(h^{\alpha-1})}.\]
Ensuite il reste juste a noter que : \[
\frac{h^{\alpha}}{h^{\alpha-1}\left|\ln(h)\right|+O(h^{\alpha-1})}=\frac{h}{\left|\ln(h)\right|+O(1)}\]
\[
=\frac{h}{\left|\ln(h)\right|}\frac{1}{1+O\left(\frac{1}{\left|\ln(h)\right|}\right)}\]
\[
=\frac{h}{\left|\ln(h)\right|}\left(1+O\left(\frac{1}{\left|\ln(h)\right|}\right)\right)\]
et comme pour $h$ assez petit $h/\ln(h)^{2}\ll h/\left|\ln(h)\right|$
on démontre la proposition 4.12.
\end{proof}
En résumant toute cette partie 4, on a bien montré le théorème 4.1.

\subsection{Quelques remarques}

Pour finir, on va donner deux remarques, la première est technique
et concerne le diamètre du compact où le théorème 4.1 est valide.
Dans la seconde remarque on tente de donner un panorama global sur
le spectre du double puits à l'aide des résultats connus sur le haut
de spectre \textbf{{[}21]}, \textbf{{[}11]} et sur le bas de spectre
\textbf{{[}20]}. Quelques tracés numériques sont aussi proposés.

\subsection*{Une remarque technique}

Dans la preuve du théorème 4.1 on a vu la nécessité technique d'avoir
supposé $\alpha\geq\frac{1}{2}$ (voir en particulier les majorants
4.5 et 4.6). Cependant, malgré cette hypothèse, le théorème 4.1 reste
assez intéressant, notamment en vu d'applications: par exemple pour
l'étude de la dynamique quantique d'un paquet d'ondes, en effet la
taille $\sqrt{h}$ est (modulo un facteur multiplicatif) la taille
d'un boule d'aire $h$, c'est à dire en physique la taille d'un quanta%
\footnote{Par exemple les états cohérents.%
} .

\subsection*{Une remarque sur le spectre global du double puits}

Pour fixer les idées, on va supposer que le double puits est symétrique
(la fonction $V$ est paire), le bas du spectre est alors uniquement
constitué de quasi-doublets%
\footnote{En fait les valeurs propres sont toutes simples, mais la présence
des deux puits induit deux spectres exponentiellement proches l'un
de l'autre; voir\textbf{ {[}20]}.%
} de valeurs propres distant l'un de l'autre de $h$. Le haut de spectre
est constitué de valeurs propres régulièrement espacées de taille
$h$ (voir \textbf{{[}21]}, \textbf{{[}11]}). Sur la figure 7 on voit
le passage du spectre quasi-double correspondant au bas du spectre
au spectre simple correspondant au haut du spectre. On voit aussi
que les valeurs propres se resserrent entre elles au passage du maximun
local et s'écartent lorsque on monte vers le haut du spectre. 

\begin{center}
\includegraphics[scale=0.45]{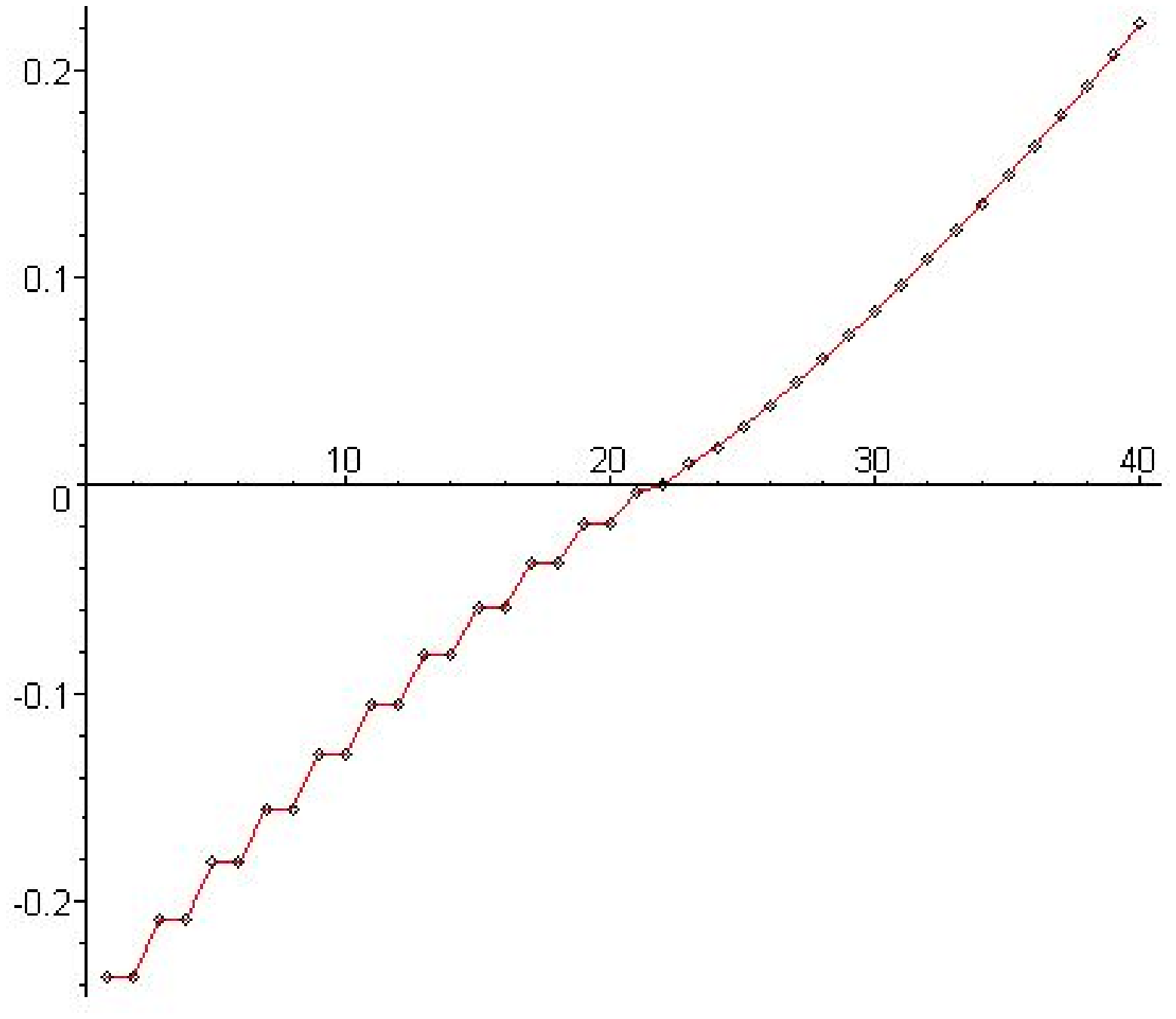}
\par\end{center}

\begin{center}
\textit{Fig. 7: Tracé numérique d'une partie du spectre du double
puits symétrique : sur l'axe des abscisses on trouve un indexage des
valeurs propres, et sur l'axe des ordonnées on trouve les valeurs
propres. }
\par\end{center}

\vspace{+0.25cm}

La figure 8 décrit la différence entres 2 valeurs propres consécutives
: en passant du bas au haut du spectre on voit que l'oscillation induite
par le phénomène de quasi-doublets diminue jusqu'à disparaître. Sur
cette même figure 7 on distingue aussi très bien le resserrement {}``logarithmique''
des valeurs propres au passage du maximun local, puis l'écartement,
lui aussi {}``logarithmique'', des valeurs propres quand on remonte
dans la haut du spectre.

\begin{center}
\includegraphics[scale=0.32]{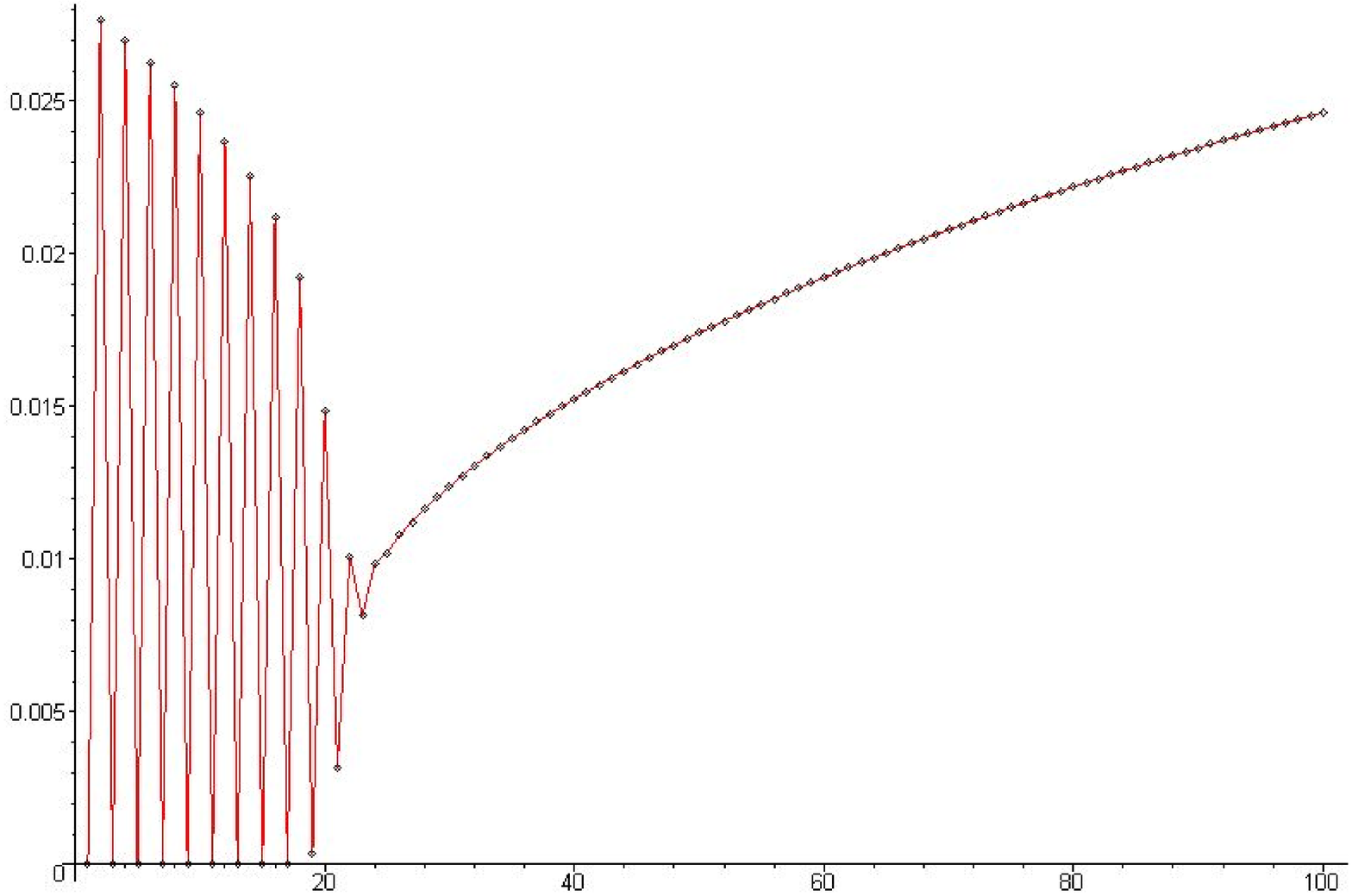}
\par\end{center}

\begin{center}
\textit{Fig. 8: Tracé numérique de la différence entres 2 valeurs
propres consécutives dans le cas du double puits symétrique : sur
l'axe des abscisses on trouve un indexage des valeurs propres, et
sur l'axe des ordonnées on trouve la différence entres 2 valeurs propres
consécutives.}
\par\end{center}

\vspace{+0.25cm}

Cette article donne le train d'union entre le bas et le haut de spectre
: même si les transitions restent mathématiquement délicates à écrire,
on peut imaginer qu'en partant du bas de spectre et en montant vers
le maximun local, la quinconce exponentiel des quasi-doublets augmente
jusqu'à apparaître%
\footnote{Il faut dire qu'en méthode semi-classique la distance exponentiellement
petite est un $O(h^{\infty})$, donc non visible. %
} clairement. Dans le même temps, toutes les valeurs propres se resserrent
(passage de la distance spectrale $h$ à $h/\left|\ln(h)\right|$).
Ensuite quand on continue de monter du maximun vers le haut de spectre,
il faut la aussi imaginer que la quinconce devient équidistante et
que les valeurs propres s'écartent (passage de la distance $h/\left|\ln(h)\right|$
à $h$ ).

\section{Annexe }

Le but de cette annexe est de montrer de manière détaillée, à l'aide
des distributions, le lemme 3.8 (voir aussi\textbf{ {[}19]}).

\subsection{Distributions tempérées holomorphes}

\begin{defn}
Soit $U$ un ouvert non vide de $\mathbb{C}$, et considérons l'application
:\[
T_{.}\,:\,\left\{ \begin{array}{cc}
U\rightarrow\mathcal{S}^{\prime}(\mathbb{R})\\
\\\lambda\mapsto T_{\lambda}.\end{array}\right.\]
On dira que $T_{\lambda}$ est holomorphe sur $U$ si et seulement
si pour tout $\varphi\in\mathcal{S}(\mathbb{R})$ la fonction $\lambda\mapsto\left\langle T_{\lambda},\varphi\right\rangle _{\mathcal{S}^{\prime},\mathcal{S}}$
est holomorphe sur $U$.
\end{defn}

\subsection{Les distributions $\left[x_{+}^{\lambda}\right]$ et $\left[x_{-}^{\lambda}\right]$ }

Soit $\lambda\in\mathbb{C}$ tels que Re$(\lambda)>-1$, on définit
alors les fonctions $x_{+}^{\lambda}$ et $x_{-}^{\lambda}$ par :
$x_{+}^{\lambda}:=\mathbf{1}_{\mathbb{R}_{+}^{*}}(x)x^{\lambda}$
et $x_{-}^{\lambda}:=\mathbf{1}_{\mathbb{R}_{-}^{*}}(x)|x|^{\lambda}$.
Comme Re$(\lambda)>-1$ on vérifie sans peine que $x_{+}^{\lambda}$
et $x_{-}^{\lambda}$ sont dans $L_{\textrm{loc}}^{1}(\mathbb{R})$
et que les distributions $\left[x_{+}^{\lambda}\right]$ et $\left[x_{-}^{\lambda}\right]$
sont holomorphes sur $\left\{ z\in\mathbb{C},\,\textrm{Re}(z)>-1\right\} $.

\begin{prop}
Les distributions $\left[x_{+}^{\lambda}\right]$ et $\left[x_{-}^{\lambda}\right]$
admettent toutes les deux un prolongement holomorphe à $\mathbb{C}-\mathbb{Z}_{-}^{*}$.
\end{prop}
\begin{proof}
Pour le moment la distribution $\left[x_{+}^{\lambda}\right]$ n'a
de sens que pour $\lambda\in\mathbb{C}$ tels que $\textrm{Re}(\lambda)>-1$.
Alors comme pour tout $\varphi\in\mathcal{S}(\mathbb{R})$ nous avons
que\[
\left\langle \left[x_{+}^{\lambda}\right],\varphi\right\rangle _{\mathcal{S}^{\prime},\mathcal{S}}={\displaystyle \int_{0}^{+\infty}}x^{\lambda}\varphi(x)\, dx\]
\[
={\displaystyle \int_{0}^{1}}x^{\lambda}\left(\varphi(x)-\varphi(0)\right)\, dx+{\displaystyle \int_{1}^{+\infty}}x^{\lambda}\varphi(x)\, dx+\frac{\varphi(0)}{\lambda+1}.\]
Il est clair que $\lambda\mapsto\frac{\varphi(0)}{\lambda+1}$ est
holomorphe sur $\mathbb{C}-\left\{ -1\right\} $, que $\lambda\mapsto{\displaystyle \int_{1}^{+\infty}}x^{\lambda}\varphi(x)\, dx$
est holomorphe sur $\mathbb{C}$, et que $\lambda\mapsto{\displaystyle \int_{0}^{1}}x^{\lambda}\left(\varphi(x)-\varphi(0)\right)\, dx$
est absolument convergente pour $\lambda\in\mathbb{C}$ tels que $\textrm{Re}(\lambda)>-2$,
ainsi l'égalité précédente est vrai pour $\lambda\in\mathbb{C}-\left\{ -1\right\} $
tels que $\textrm{Re}(\lambda)>-2$ , on vient donc de définir $\left[x_{+}^{\lambda}\right]$
pour $\lambda\in\mathbb{C}-\left\{ -1\right\} $ tels que $\textrm{Re}(\lambda)>-2$.
Itérons ce procédé: en écrivant que \[
\left\langle \left[x_{+}^{\lambda}\right],\varphi\right\rangle _{\mathcal{S}^{\prime},\mathcal{S}}={\displaystyle \int_{0}^{+\infty}}x^{\lambda}\varphi(x)\, dx\]
\[
={\displaystyle \int_{0}^{1}}x^{\lambda}\left(\varphi(x)-\varphi(0)-x\varphi\prime(0-...-\frac{x^{n-1}}{(n-1)!}\varphi^{(n-1)}(0)\right)\, dx\]
\[
+{\displaystyle \int_{1}^{+\infty}}x^{\lambda}\varphi(x)\, dx+{\displaystyle \sum_{k=1}^{n}}\frac{\varphi^{(k-1)}(0)}{(k-1)!(\lambda+k)}.\]
Avec le même procédé qu'avant, on peut définir la distribution $\left[x_{+}^{\lambda}\right]$
pour $\lambda\in\mathbb{C}-\left\{ -1,-2,\cdots,-n\right\} $ et tels
que $\textrm{Re}(\lambda)>-n-1$. Ainsi par récurrence on peut définir
la distribution $\left[x_{+}^{\lambda}\right]$ pour tout $\lambda\in\mathbb{C}-\left\{ \mathbb{Z}_{-}^{*}\right\} $. 
\end{proof}

\subsection{Les distributions $\left(x+i0\right)^{\lambda}$ et $\left(x-i0\right)^{\lambda}$}

Soit $\lambda\in\mathbb{C}$ et pour tout $(x,y)\in\mathbb{R}^{2}$
on définit la fonction $(x+iy)^{\lambda}$ pour tout $(x,y)\in\mathbb{R}^{2}$
par :\[
(x+iy)^{\lambda}:=e^{\lambda\ln(x+iy)}=e^{\lambda\ln|x^{2}+y^{2}|+\lambda i\arg(x+iy)}.\]
Ainsi comme $z=x+iy\mapsto z^{\lambda}=(x+iy)^{\lambda}$ est holomorphe
sur $U_{\pi}:=\mathbb{C}-\mathbb{R}_{-}$, on va s'intéresser aux
limites quand on s'approche de l'axe des réels par le haut et par
le bas de l'axe; on a simplement que :\[
(x+i0)^{\lambda}=\lim_{y\rightarrow0^{+}}(x^{2}+y^{2})^{\frac{\lambda}{2}}e^{\lambda i\arg(x+iy)}=\left\{ \begin{array}{cc}
x^{\lambda}\;\; si\;\; x\geq0\\
\\|x|^{\lambda}e^{\lambda i\pi}\;\; si\;\; x\leq0\end{array}\right.\]

et\[
(x-i0)^{\lambda}=\lim_{y\rightarrow0^{-}}(x^{2}+y^{2})^{\frac{\lambda}{2}}e^{\lambda i\arg(x+iy)}=\left\{ \begin{array}{cc}
x^{\lambda}\;\; si\;\; x\geq0\\
\\|x|^{\lambda}e^{-\lambda i\pi}\;\; si\;\; x\leq0.\end{array}\right.\]
Ces deux nouvelles fonctions $(x+i0)^{\lambda}$ et $(x-i0)^{\lambda}$
sont bien définies sur tout $\mathbb{C}$ et donc pour tout $\lambda\in\mathbb{C}$
tels que Re$(\lambda)>-1$ ces fonctions peuvent aussi s'écrirent
:\[
(x+i0)^{\lambda}=x_{+}^{\lambda}+e^{\lambda i\pi}x_{-}^{\lambda}\textrm{ et }(x-i0)^{\lambda}=x_{-}^{\lambda}+e^{-\lambda i\pi}x_{+}^{\lambda}.\]
Ainsi pour tout $\lambda\in\mathbb{C}$ tels que Re$(\lambda)>-1$,
au sens des distributions nous avons que :\[
(x+i0)^{\lambda}=\left[x_{+}^{\lambda}\right]+e^{\lambda i\pi}\left[x_{-}^{\lambda}\right]\textrm{ et }(x-i0)^{\lambda}=\left[x_{-}^{\lambda}\right]+e^{-\lambda i\pi}\left[x_{+}^{\lambda}\right].\]
Avec la proposition 5.2 on peut définir un prolongement holomorphe
de $\left(x+i0\right)^{\lambda}$ et de $\left(x-i0\right)^{\lambda}$
à $\mathbb{C}-\mathbb{Z}_{-}^{*}$.

\subsection{Calcul de la transformée de Fourier de $\left[x_{+}^{\lambda}\right]$
et de $\left[x_{-}^{\lambda}\right]$ }

On va démontrer le lemme 3.9 : soit $\lambda\in\mathbb{C}$ tels que
$\textrm{Re}(\lambda)\in\left]0,1\right[$. Soit $\tau>0$ pour tout
$x\in\mathbb{R}$ nous avons 

\textit{\[
\mathcal{F}_{1}\left(t_{+}^{\lambda}e^{-\tau t}\right)(x)=\frac{1}{\sqrt{2\pi}}{\displaystyle \int_{0}^{+\infty}t^{\lambda}e^{-\tau t}e^{-ixt}\, dt}\]
\begin{equation}
=\frac{1}{\sqrt{2\pi}}{\displaystyle \int_{0}^{+\infty}t^{\lambda}e^{ist}\, dt}\label{eq:}\end{equation}
}où on a posé $s:=-x+\tau i$.

Comme Im(s)$>0,$ l'intégrale (5.1) converge absolument et on peut
voir avec le théorème de convergence dominée que la distribution $\left[x_{+}^{\lambda}e^{-\tau x}\right]$
converge dans $\mathcal{S}^{\prime}(\mathbb{R})$ vers la distribution
$\left[x_{+}^{\lambda}\right]$ quand $\tau\rightarrow0$, ainsi par
continuité de la transformée de Fourier la distribution $\mathcal{F}_{1}\left(\left[x_{+}^{\lambda}e^{-\tau x}\right]\right)$
converge dans $\mathcal{S}^{\prime}(\mathbb{R})$ vers la distribution
$\mathcal{F}_{1}\left(\left[x_{+}^{\lambda}\right]\right)$ quand
$\tau\rightarrow0$.

On va maintenant calculer l'intégrale (5.1) : en faisant le changement
de variable $u:=-ist$, avec $s=-x+\tau i$, où $\tau>0$, ainsi $\arg(s)\in\left]0,\pi\right[$
on a:\[
{\displaystyle \int_{0}^{+\infty}t^{\lambda}e^{ist}\, dt}=\left(\frac{i}{s}\right)^{\lambda+1}{\displaystyle \int_{L}u^{\lambda}e^{-u}\, du}\]
où $L$ est la demi-droite partant de $0$ et d'angle $\arg(-is)$.
Notons bien que $\arg(-ist)=\arg(s)-\frac{\pi}{2}$, donc $\arg(-ist)\in\left]-\frac{\pi}{2},\frac{\pi}{2}\right[.$
Maintenant montrons que :\[
{\displaystyle \int_{L}u^{\lambda}e^{-u}\, du}={\displaystyle \int_{0}^{+\infty}x^{\lambda}e^{-x}\, dx.}\]
Pour cela soit $0<\epsilon<R$, et considérons le lacet orienté $\gamma_{\epsilon,R}$
du plan complexe défini par (voir Figure suivante) :

\begin{center}
\includegraphics[scale=0.5]{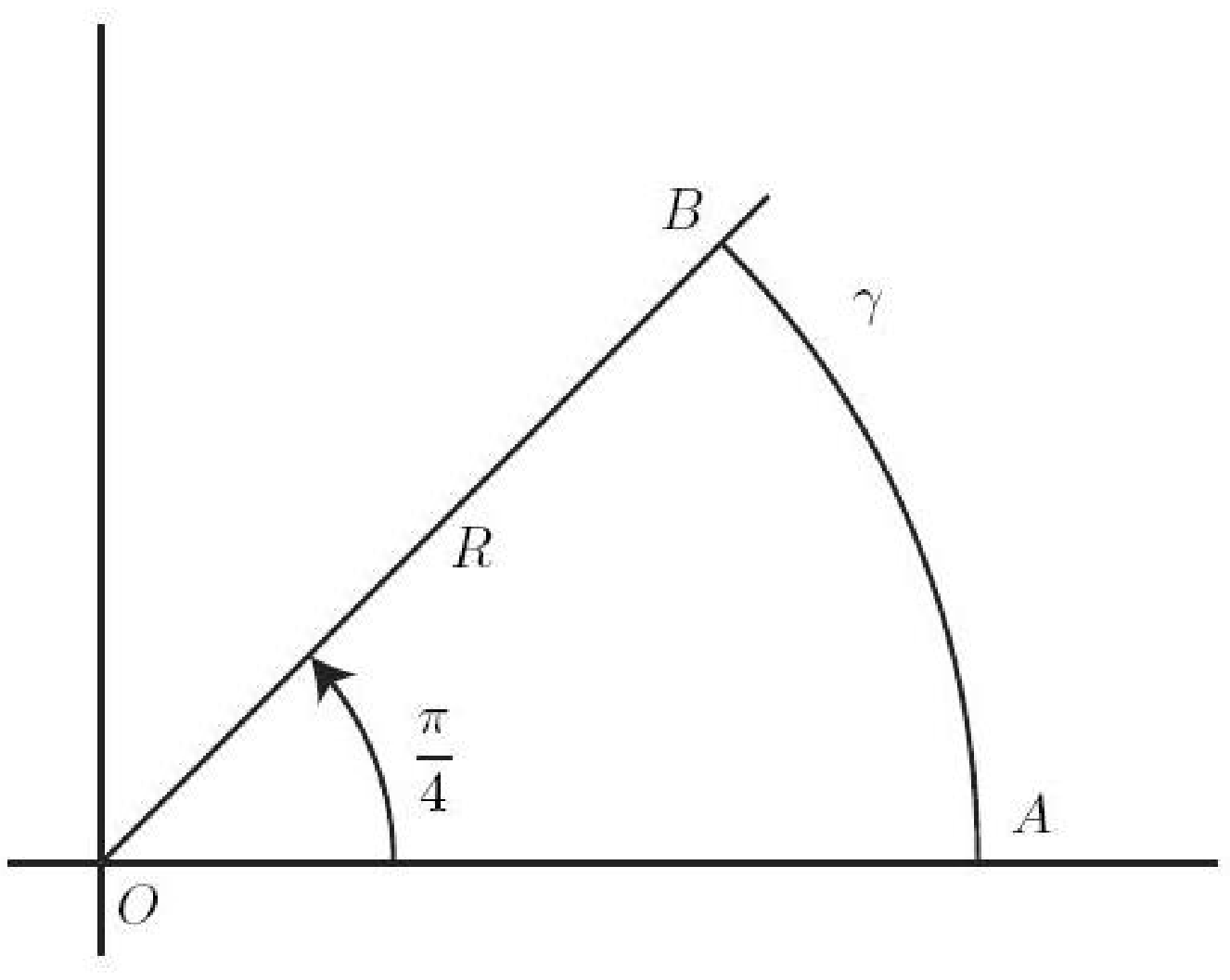}
\par\end{center}

\begin{center}
\textit{Fig. 9: Chemin d'intégration $\gamma_{\epsilon,R}:=L_{\epsilon,R}\cup C_{R}\cup\left[R,\epsilon\right]\cup C_{\epsilon}$.}
\par\end{center}

\vspace{+0.25cm}

Alors comme $f\,:\, z\mapsto z^{\lambda}e^{-z}=e^{\lambda\ln(z)}e^{-z}$
est holomorphe sur tout ouvert de $\left\{ z\in\mathbb{C},\,\textrm{Re}(z)>0\right\} $
par le théorème de Cauchy on a d'une part que :\[
{\displaystyle \int_{\gamma_{\epsilon,R}}f(z)\, dz}=0\]
et d'autre part en décomposant le lacet on obtient l'égalité suivante
:

\textit{\begin{equation}
{\displaystyle \int_{\gamma_{\epsilon,R}}f(z)\, dz}={\displaystyle \int_{L_{\epsilon,R}}f(z)\, dz}+{\displaystyle \int_{C_{R}}f(z)\, dz}+{\displaystyle \int_{R}^{\epsilon}f(x)\, dx}+{\displaystyle \int_{C_{\epsilon}}f(z)\, dz.}\label{eq:}\end{equation}
}Or avec le changement de variable $z:=Re^{i\theta}$ dans la seconde
intégrale de (5.2) nous avons

\[
{\displaystyle \int_{C_{R}}f(z)\, dz}=-i{\displaystyle \int_{0}^{\arg(-is)}R^{\lambda+1}e^{i\lambda\theta}e^{-Re^{i\theta}}\, d\theta}\]
or \[
\left|R^{\lambda+1}e^{i\lambda\theta}e^{-Re^{i\theta}}\right|\leq R^{Re(\lambda)+1}e^{-R\cos(\theta)}\]
donc, comme $\arg(-is)\in\left]-\frac{\pi}{2},\frac{\pi}{2}\right[$,
pour tout $\theta\in\left[0,\arg(-is)\right],\,\cos(\theta)>0$ ,
ainsi

\[
{\displaystyle \lim_{R\rightarrow+\infty}}R^{Re(\lambda)+1}e^{-R\cos(\theta)}=0\]
d'où par convergence dominée sur un compact :\[
{\displaystyle \lim_{R\rightarrow+\infty}}{\displaystyle \int_{C_{R}}f(z)\, dz}=0.\]
Avec le changement de variable $z:=\epsilon e^{i\theta}$ dans la
quatrième intégrale de (5.2) nous avons

\[
{\displaystyle \int_{C_{\epsilon}}f(z)\, dz}=i{\displaystyle \int_{0}^{\arg(-is)}\epsilon^{\lambda+1}e^{i\lambda\theta}e^{-\epsilon e^{i\theta}}\, d\theta}\]
donc\[
\left|{\displaystyle \int_{C_{\epsilon}}f(z)\, dz}\right|\leq\epsilon^{Re(\lambda)+1}|\arg(-is)|\]
et comme Re$(\lambda)+1>0$

\[
{\displaystyle \lim_{\epsilon\rightarrow0^{+}}}\epsilon^{Re(\lambda)+1}|\arg(-is)|=0\]
on a 

\[
{\displaystyle \lim_{\epsilon\rightarrow0^{+}}}{\displaystyle \int_{C_{\epsilon}}f(z)\, dz}=0.\]
Enfin par le théorème de Cauchy, et comme (5.1) converge absolument,
en faisant tendre $\epsilon\rightarrow0$ et $R\rightarrow+\infty$,
l'égalité est valable pour tout $\lambda\in\mathbb{C}$ tels que $-1<$Re$(\lambda)<0$\[
0={\displaystyle \int_{L}u^{\lambda}e^{-u}\, dt}+{\displaystyle \int_{+\infty}^{0}x^{\lambda}e^{-x}\, dx}\]
donc 

\[
{\displaystyle \int_{L}u^{\lambda}e^{-u}\, dt}={\displaystyle \int_{0}^{+\infty}x^{\lambda}e^{-x}\, dx}=\Gamma(\lambda+1)\]
ie :\[
\mathcal{F}_{1}\left(t_{+}^{\lambda}e^{-\tau t}\right)(x)=\frac{1}{\sqrt{2\pi}}\left(\frac{i}{s}\right)^{\lambda+1}\Gamma(\lambda+1)\]
\[
=\frac{1}{\sqrt{2\pi}}\left(\frac{e^{i\frac{\pi}{2}}}{s}\right)^{\lambda+1}\Gamma(\lambda+1)\]
\[
=\frac{1}{\sqrt{2\pi}}\frac{ie^{i\frac{\pi}{2}\lambda}}{(-x+i\tau)^{\lambda+1}}\Gamma(\lambda+1).\]
Maintenant passons à la limite $(\tau\rightarrow0)$ dans $\mathcal{S}^{\prime}(\mathbb{R})$
on obtient :\[
\mathcal{F}_{1}\left(\left[t_{+}^{\lambda}\right]\right)(x)=\frac{1}{\sqrt{2\pi}}\frac{ie^{i\frac{\pi}{2}\lambda}}{(-x+i.0)^{\lambda+1}}\Gamma(\lambda+1).\]
Alors comme $\lambda\mapsto\frac{1}{\sqrt{2\pi}}\frac{ie^{i\frac{\pi}{2}\lambda}}{(-x+i.0)^{\lambda+1}}\Gamma(\lambda+1)$
est holomorphe sur $\mathbb{C}-\mathbb{Z}_{-}^{*}$ et que $\lambda\mapsto\mathcal{F}_{1}\left(\left[t_{+}^{\lambda}\right]\right)$
est holomorphe sur $\mathbb{C}-\mathbb{Z}^{*}$ , par un prolongement
holomorphe la précédente égalité reste vraie sur $\mathbb{C}-\mathbb{Z}^{*}$.
Et donc pour tout $\lambda\in\mathbb{C}-\mathbb{Z}^{*}$ nous avons

\[
\mathcal{F}_{1}\left(\left[t_{+}^{\lambda}\right]\right)(x)=\frac{1}{\sqrt{2\pi}}\frac{ie^{i\frac{\pi}{2}\lambda}}{(-x+i.0)^{\lambda+1}}\Gamma(\lambda+1)\]
et donc\[
\mathcal{F}_{1}\left(\left[t_{+}^{\lambda}\right]\right)(-x)=\frac{1}{\sqrt{2\pi}}\frac{ie^{i\frac{\pi}{2}\lambda}}{(x+i.0)^{\lambda+1}}\Gamma(\lambda+1)\]
ensuite comme

\[
(x+i.0)^{-\lambda-1}=\left[x_{+}^{-\lambda-1}\right]-e^{-\lambda i\pi}\left[x_{-}^{-\lambda-1}\right]\]
pour tout $\lambda\in\mathbb{C}-\mathbb{Z}^{*}$on a 

\[
\mathcal{F}_{1}\left(\left[t_{+}^{\lambda}\right]\right)(-x)=\frac{ie^{i\frac{\pi}{2}\lambda}\Gamma(\lambda+1)}{\sqrt{2\pi}}\left(\left[x_{+}^{-\lambda-1}\right]-e^{-\lambda i\pi}\left[x_{-}^{-\lambda-1}\right]\right)\]
d'où\[
\mathcal{F}_{1}\left(\left[t_{+}^{\lambda}\right]\right)(x)=\frac{ie^{i\frac{\pi}{2}\lambda}\Gamma(\lambda+1)}{\sqrt{2\pi}}\left(\left[x_{-}^{-\lambda-1}\right]-e^{-\lambda i\pi}\left[x_{+}^{-\lambda-1}\right]\right).\]

\vspace{1cm}

\hspace{-0.5cm}\textbf{\large Olivier Lablée}{\large \par}

\hspace{-0.5cm}

\hspace{-0.5cm}Université Grenoble 1-CNRS

\hspace{-0.5cm}Institut Fourier

\hspace{-0.5cm}UFR de Mathématiques

\hspace{-0.5cm}UMR 5582

\hspace{-0.5cm}BP 74 38402 Saint Martin d'Hères 

\hspace{-0.5cm}mail: \textcolor{blue}{lablee@ujf-grenoble.fr}

\hspace{-0.5cm}http://www-fourier.ujf-grenoble.fr/\textasciitilde{}lablee/
\end{document}